\documentclass[11pt]{amsart}

\usepackage[utf8]{inputenc}
\usepackage{lmodern}
\usepackage{microtype}
\usepackage{amsmath,amssymb,amsthm,mathtools}
\usepackage{enumitem}
\usepackage{booktabs}
\usepackage{graphicx}
\usepackage{placeins}
\usepackage{tikz}
\usetikzlibrary{
  arrows.meta,
  backgrounds,
  calc,
  fit,
  matrix,
  positioning,
  shapes.geometric
}
\usepackage[margin=1.1in]{geometry}
\usepackage[colorlinks=true,linkcolor=blue,citecolor=blue,urlcolor=blue]{hyperref}

\definecolor{oqrwblue}{RGB}{44,95,145}
\definecolor{oqrwbluefill}{RGB}{226,238,249}
\definecolor{oqrwgreen}{RGB}{45,120,83}
\definecolor{oqrwgreenfill}{RGB}{227,244,235}
\definecolor{oqrwred}{RGB}{166,70,60}
\definecolor{oqrwredfill}{RGB}{250,233,230}
\definecolor{oqrwgold}{RGB}{167,119,31}
\definecolor{oqrwgoldfill}{RGB}{252,243,220}
\definecolor{oqrwgray}{RGB}{80,87,94}
\definecolor{oqrwgrayfill}{RGB}{240,242,244}

\tikzset{
  oqrw arrow/.style={-{Latex[length=2.2mm]},thick,draw=oqrwgray},
  oqrw dashed arrow/.style={-{Latex[length=2.2mm]},thick,dashed,draw=oqrwgray},
  oqrw base/.style={
    draw=oqrwgray,
    rounded corners=2pt,
    thick,
    align=center,
    inner sep=5pt,
    font=\small
  },
  oqrw blue/.style={oqrw base,draw=oqrwblue,fill=oqrwbluefill},
  oqrw green/.style={oqrw base,draw=oqrwgreen,fill=oqrwgreenfill},
  oqrw red/.style={oqrw base,draw=oqrwred,fill=oqrwredfill},
  oqrw gold/.style={oqrw base,draw=oqrwgold,fill=oqrwgoldfill},
  oqrw gray/.style={oqrw base,fill=oqrwgrayfill}
}

\allowdisplaybreaks
\numberwithin{equation}{section}

\theoremstyle{plain}
\newtheorem{theorem}{Theorem}[section]
\newtheorem{proposition}[theorem]{Proposition}
\newtheorem{lemma}[theorem]{Lemma}
\newtheorem{corollary}[theorem]{Corollary}
\newtheorem{assumption}[theorem]{Assumption}
\theoremstyle{definition}
\newtheorem{definition}[theorem]{Definition}
\newtheorem{example}[theorem]{Example}
\theoremstyle{remark}
\newtheorem{remark}[theorem]{Remark}

\newcommand{\C}{\mathbb C}
\newcommand{\R}{\mathbb R}
\newcommand{\Z}{\mathbb Z}
\newcommand{\E}{\mathbb E}
\newcommand{\PP}{\mathbb P}
\newcommand{\cH}{\mathcal H}
\newcommand{\cK}{\mathcal K}
\newcommand{\cL}{\mathcal L}
\newcommand{\cR}{\mathcal R}

\newcommand{\algA}{\mathfrak A}
\newcommand{\Tr}{\operatorname{Tr}}
\newcommand{\Ran}{\operatorname{Ran}}
\newcommand{\supp}{\operatorname{supp}}
\newcommand{\rank}{\operatorname{rank}}
\newcommand{\spec}{\operatorname{spec}}
\newcommand{\diag}{\operatorname{diag}}
\newcommand{\norm}[1]{\left\lVert #1\right\rVert}
\newcommand{\abs}[1]{\left\lvert #1\right\rvert}
\newcommand{\ip}[2]{\left\langle #1,#2\right\rangle}

\title[Ballistic speed and recurrence for OQRWs]
{Ballistic Speed and Potential-Theoretic Recurrence\\
for Homogeneous Open Quantum Random Walks}
\author{Ameur Dhahri}
\address{Dipartimento di Matematica, Politecnico di Milano,
Piazza Leonardo da Vinci 32, I-20133 Milan, Italy}
\email{ameur.dhahri@polimi.it}

\author{Chul Ki Ko}
\address{University College, Yonsei University, 85 Songdogwahak-ro,
Yeonsu-gu, Incheon 21983, Korea}
\email{kochulki@yonsei.ac.kr}

\author{Farrukh Mukhamedov}
\address{Department of Mathematical Sciences, United Arab Emirates
University, P.O. Box No.~15551, Al Ain, Abu Dhabi, United Arab Emirates}
\email{far75m@yandex.ru; farrukh.m@uaeu.ac.ae}

\author{Hyun Jae Yoo}
\address{Department of Applied Mathematics and Institute for Integrated
Mathematical Sciences, Hankyong National University, 327 Jungang-ro,
Anseong-si, Gyeonggi-do 17579, Korea}
\email{yoohj@hknu.ac.kr}
\thanks{Corresponding author: Farrukh Mukhamedov}
\date{}


\hypersetup{
  pdftitle={Ballistic Speed and Potential-Theoretic Recurrence for Homogeneous Open Quantum Random Walks},
  pdfauthor={Ameur Dhahri, Chul Ki Ko, Farrukh Mukhamedov, Hyun Jae Yoo},
  pdfsubject={Open quantum random walks},
  pdfkeywords={OQRW, TOM, periodic local limit theorem, Green function, potential kernel,
  recurrence, transience, ballistic speed}
}

\begin{document}

\begin{abstract}
Our study connects root-mean-square ballistic transport to potential-theoretic recurrence for finite-range open quantum random walks.  In homogeneous walks with primitive local channels, finitely many simple periodic Fourier peripheral eigenvalues, no nonzero stationary Fourier mode, and a nondegenerate quadratic spectral term, we prove a periodic uniform local limit theorem and recover exponential finite-set return bounds for nonzero drift.  The centered case yields strong Green-function asymptotics in at least three dimensions, as well as potential-kernel asymptotics in one and two dimensions.  According to these assumptions, the RMS ballistic speed equals the drift norm, nonzero drift indicates transience, and centered walks are recurring in effective dimensions one and two but transitory in higher dimensions.  The low-dimensional finding shows a recurrence of the origin projection in TOM.  For reducible walks, an explicit harmonic \(h\)-transform converts each absorption component to an OQRW, providing a detailed breakdown of Green occupation potentials.  The squared RMS speed is the absorption-weighted mean of squared component drifts, but the reduced drift-dimension classification also needs specific component return estimations.  A centered noncommuting family validates the fundamental spectral assumptions in all dimensions and provides explicit potential constants.  Exact finite traps and sparse reflecting barriers provide a complementary nonhomogeneous method for zero speed and TOM recurrence.
\end{abstract}

\maketitle

\section{Introduction}

Quantum random walks are noncommutative models of transport that link spatial motion to an internal quantum degree of freedom.  According to the unitary theory, this linkage occurs by reversible coin-position development.  Open quantum random walks (OQRWs), developed by Attal--Petruccione--Sabot--Sinayskiy \cite{APSS12}, replace the unitary mechanism with entirely positive operations, making them suitable for dissipative dynamics, repetitive interactions, and quantum systems with external noise.  The position variable remains lattice-valued, but the probability of the next displacement is determined by the current internal density operator.  This trait explains both their versatility and technological challenges.

An OQRW is not a traditional Markov chain on the position space.
The quantum trajectory \((X_n,\rho_n)\) combines the position and normalized internal state to construct a Markov process. The marginal process \((X_n)\) does not necessarily need to be Markovian.  Basic questions for a random walk must consider the development of the internal state.  Limit theorems show that drift and covariance are dictated by the spectral and ergodic features of an auxiliary quantum channel, not a fixed scalar transition kernel.  Early limit results were obtained by Konno--Yoo \cite{KonnoYoo13}. Laws of large numbers and central limit theorems were later established on integer lattices by Attal--Guillotin-Plantard--Sabot \cite{AGPS15}, on crystal lattices by Ko--Konno--Segawa--Yoo \cite{KKSY19}, and, without irreducibility of the local channel, by Carbone--Girotti--Hernandez \cite{CGH22}.

Carbone--Pautrat \cite{CarbonePautrat15} proved a central limit theorem and a large-deviation principle for finite-dimensional homogeneous OQRWs with irreducible local channels, utilizing real channel deformations.  Exponential concentration around the drift implies finite-set transience when the drift is not zero.  We provide a uniform spectral argument to support our local theorem.  This paper focuses on periodic local probabilities, quantitative Green-function and potential-kernel asymptotics, and their impact on the transition operation matrix (TOM) recurrence.  These need knowledge beyond macroscopic boundaries and significant variations.

Ballistic transport is a well-known concept in the unitary quantum-walk literature.  Konno \cite{Konno05} found non-Gaussian limits for position divided by time in one-dimensional defined walks, whereas Grimmett-Janson-Scudo \cite{GrimmettJansonScudo04} verified ballistic weak-limit theorems using Fourier techniques in one and higher dimensions.  Ahlbrecht--Vogts--Werner--Werner \cite{AhlbrechtEtAl11} identified the limiting velocity distribution for translation-invariant finite-range unitary walks using dispersion band group velocities.
The position distribution is collected following unitary development, without any intermediate observations.  The ballistic limit may be non-degenerate, resulting in spatial dispersion on the scale \(n\) even when the mean velocity is zero.  The symmetric Hadamard walk in \cite{Konno05} is a simple example.  Positive RMS ballistic speed does not always indicate nonzero deterministic drift in the unitary situation.

Noise may impact the behavior of the model and its mixing assumptions, resulting in different conclusions.  Brun-Carteret-Ambainis \cite{BrunCarteretAmbainis03} computed position moments for a coined walk with coin decoherence and found an asymptotically linear variance for the decohering model.  Ahlbrecht et al. (2011) propose deterministic ballistic drift and momentum-dependent Gaussian fluctuation laws for Markov-controlled random unitary coins. However, averaging across momentum does not always result in a single Gaussian.  These models differ from OQRWs in that OQRW evolution retains position-diagonal states and follows a measured trajectory, whereas random-unitary or coin-only noise may maintain spatial coherences.

The TOM formalism provides an operator-theoretic framework for comparing these dynamics to traditional Markov chains.  A TOM, as defined by Gudder \cite{Gudder08}, substitutes the elements of a stochastic matrix with entirely positive mappings, and conserves total trace instead of total probability.  OQRWs are a natural Kraus-type subclass.  Using the dual Heisenberg evolution results in a normal unital quantum Markov operator for the diagonal observable algebra.
This perspective has been used to link OQRWs with quantum Markov chains \cite{DhahriKoYoo19,DhahriMukhamedov19,DM19} and to create a potential theory for QRWs with general TOMs \cite{DhahriEtAl26}. Related potential-theoretic constructions for quantum Markov chains can be found in \cite{DhahriFagnola23}.

The choice of quantum formalism has a particularly strong impact on recurring events.
Measurement-based techniques monitor a site or subspace and analyze the earliest returns of halted evolution.  Occupation-based techniques calculate the likelihood of locating a walker in a target set at certain times.
Operator-theoretic techniques investigate the potential of a projection under unmonitored Heisenberg evolution.  Although these concepts are connected in particular models, they do not always align due to the impact of monitoring on dynamics.  In \cite{CarvalhoGuidiLardizabal17} it was explored site recurrence for open and unitary walks on the line, whereas \cite{JacqLardizabal21} created stricter criteria for homogenous OQRWs on the line.  Recurrence has been studied using quantum Markov chains \cite{DhahriKoYoo19,DhahriMukhamedov19} and Schur-function and splitting approaches \cite{GrunbaumLardizabalVelazquez20}.  The TOM potential in this study belongs to the unmonitored operator-theoretic side of the picture, whereas our spatial Green series is directly linked to the measured position law.

The second problem is reducibility.  The local channel may have many recurrent enclosures, resulting in distinct drifts and covariances based on the initial internal state.  In \cite{CarbonePautrat16} it has been methodically explored the reducibility, period, and ergodic features of OQRWs. the generalized limit theorem of \cite{CGH22} shows
that the asymptotic position law may be a mixture of component laws rather
than a single Gaussian regime.  Related OQRW--quantum-Markov-chain models on
trees exhibit further phase-transition and recurrence phenomena
\cite{MukhamedovSouissiHamdi22,MukhamedovSouissiHamdiAndolsi23}.  These results
make clear that a single averaged drift need not capture the transport of a
reducible walk.

Existing theories leave a distinct gap.  General TOM potential theory establishes operator conditions for recurrence, transience, and reducibility \cite{DhahriEtAl26}, but does not include Fourier drift, covariance, or effective lattice in a homogeneous OQRW.
While laws of large numbers and central limit theorems may characterize macroscopic scaling, they cannot regulate the local probabilities that make up the Green potential.  Weak convergence is insufficient to determine the convergence of \(\sum_n\PP(X_n\in K)\).
The current approach connects two levels: Fourier spectrum analysis provides uniform local estimates, TOM duality translates these into potential-theoretic assertions, and absorption theory provides accurate potential identities for reducible dynamics.  To achieve a reducible classification based on component drifts and lattice rankings, component return estimates must be supplied in Assumption~\ref{ass:component-return}.

Against this background, the present paper asks how macroscopic ballistic
transport is related to potential-theoretic recurrence for homogeneous
finite-range OQRWs.  The transport quantity considered here is the
root-mean-square (RMS) ballistic speed
\[
  v_{\mathrm{bal}}
  =\lim_{n\to\infty}\frac{1}{n}
    \bigl(\E|X_n|^2\bigr)^{1/2},
\]
when the limit exists.  In a one-component law-of-large-numbers regime this
quantity contains no information beyond the drift: bounded jumps imply
\(v_{\mathrm{bal}}=|m|\).  The recurrence question is different.
For a specified initial state it is governed by the Green occupation series
\[
  \sum_{n\geq0}\PP(X_n\in K)
  =\Tr\!\left(\rho^{(0)}\sum_{n\geq0}T^n(P_K)\right),
\]
where \(T\) is the Heisenberg-picture transition operation matrix (TOM)
operator and \(P_K\) is the position projection of a finite set \(K\).
This identity links measured spatial occupation to a TOM potential, but it
does not reduce recurrence to a second-moment growth rate.  Drift, covariance,
effective lattice dimension, internal mixing, and access to recurrent
components all play different roles.

Recent exponential-tail findings provide a further important comparison.
Cedzich--Joye--Werner--Werner \cite{CedzichEtAl26} use analytic deformations to limit propagation tails in translation-invariant unitary lattice dynamics with compactly supported starting states.  The positive large-deviation rate applies outside the closed convex hull of the group-velocity zone, but does not cover nonconvex gaps inside the hull.  Our real deformation of the OQRW channel provides exponential probability for accessing a specified finite set when the unique law-of-large-numbers velocity is nonzero.  Our nondegenerate spectral assumptions lead to polynomial centered return estimations, necessitating a local Fourier analysis (see below).  A positive unitary RMS speed alone cannot provide the same conclusion since zero may still be inside the velocity range.

The difference between these transport findings lies in both the dynamics and the quantities being assessed.  The ballistic weak limits in \cite{Konno05,GrimmettJansonScudo04,AhlbrechtEtAl11} characterize macroscopic position distributions, while the noise analysis in \cite{BrunCarteretAmbainis03,AhlbrechtEtAl11} addresses drift and spreading, and \cite{CedzichEtAl26} manages propagation tails.  This study presents uniform local return estimates, summability, and spatial Green-function and potential-kernel asymptotics for OQRWs with explicit spectral assumptions.  It also links these estimates to TOM recurrence and precise absorption-component decompositions.  Fourier perturbation and exponential deformation are not new methodologies, but rather common instruments that contribute to local and potential-theoretical findings.  Remark ~\ref{rem:ballistic-velocity-laws} clarifies the distinction between RMS transport and centered ballistic spreading.

More precisely, the contributions are as follows.
\begin{enumerate}[label=\textup{(N\arabic*)},leftmargin=2.8em]
\item We prove a \emph{periodic uniform} Fourier local limit theorem under
explicit primitive, finite-peripheral-set, and nondegeneracy assumptions.
Every unit-modulus Fourier point contributes an explicit oscillatory term,
so no strong-aperiodicity hypothesis is imposed.  The resulting periodic
mask has mean one and yields sharp cumulative return asymptotics in effective
dimensions one and two.  A real Laplace deformation, treated in the same
spectral framework, yields exponential return estimates for every finite set
when the drift is nonzero.

\item We identify the correct spatial parameter as the effective lattice
generated by the active jumps, rather than the ambient lattice.  Combining
the local theorem with a law of large numbers proves both
\(v_{\mathrm{bal}}=|m|\) and the sharp dichotomy
\[
\begin{array}{c|c}
  \text{regime} & \text{spatial behavior}\\
  \hline
  m\neq0 & \text{transient in every effective dimension},\\
  m=0,\ r\leq2 & \text{recurrent at the origin},\\
  m=0,\ r\geq3 & \text{transient}.
\end{array}
\]
Here \(r\) is the rank of the effective lattice.  This separates covariance
degeneracy caused by a lower-dimensional support from genuine
higher-dimensional transience.  In addition, the local theorem upgrades
state-dependent recurrence to TOM recurrence of \(P_0\): for every nonzero
internal vector, the potential of \(P_0\) diverges on the corresponding
vector state.  In the centered case, Fourier resolvent analysis gives the quantitative refinements
\[
  G_{\rho_0}(x)\asymp
  (x^T\Sigma^{-1}x)^{1-r/2}\quad(r\geq3),
\]
with the exact leading constant, and the corresponding linear and logarithmic
potential-kernel asymptotics for \(r=1\) and \(r=2\).

\item For reducible OQRWs, we pass from the absorption measures of
Carbone--Girotti--Hernandez to exact potential identities.  An explicit
harmonic \(h\)-transform realizes every accessible absorption component as a
normalized OQRW, so its Fourier family and effective lattice are defined
without ambiguity.  Green occupation potentials decompose componentwise,
whereas the generalized law of large numbers gives the distinct quadratic
identity
\[
  v_{\mathrm{bal}}^2
  =\sum_{\alpha}a_\alpha|m_\alpha|^2.
\]
For a fixed finite target, its Green series diverges precisely when at least
one accessible component Green series diverges.  The drift--dimension
criterion for these components is conditional on
Assumption~\ref{ass:component-return}; the exact potential decomposition
does not need that assumption.  A sufficient condition is that every
accessible transformed Fourier family satisfy
Assumption~\ref{ass:fourier}.  Component velocity limits alone do not imply
the required local return estimates.  The speed identity also explains why
the RMS speed need not equal the norm of the averaged drift.

\item We show by explicit models that the separation is two-sided: zero RMS
speed occurs in both recurrent and transient primitive walks, and positive RMS
speed can coexist with recurrence in a reducible walk.  In particular,
Example~\ref{ex:centered-noncommuting} verifies all the spectral hypotheses
for a centered family with noncommuting Kraus operators and
internal-state-dependent jump probabilities in every effective dimension;
the Green and potential-kernel constants are explicit.  Finally, the TOM
compression identity for finite invariant sets yields a nonhomogeneous
mechanism for zero speed and projection recurrence, including partitions by
sparse reflecting barriers of unbounded spacing.
\end{enumerate}

The identities \(v_{\mathrm{bal}}=|m|\) and
\(v_{\mathrm{bal}}^2=\sum_\alpha a_\alpha|m_\alpha|^2\) are moment
consequences of the corresponding velocity limits and bounded increments;
they are not asserted as new weak-limit theorems.  Their role is to make
the comparison with Green potentials explicit within the local-to-global
classification above.

We use two notions of recurrence and keep them distinct.  Most of the paper
concerns state-dependent spatial recurrence, expressed by divergence of an
occupation series.  We also use recurrence of a projection in the TOM sense,
which requires its potential to diverge on every nonzero vector in the
projection range.  For \(P_0\) in centered effective dimensions one and two,
the periodic local theorem establishes this stronger property; exact traps
provide a second mechanism.  The distinction matters: a finite invariant
region implies recurrence of its set projection and zero RMS speed, but not
recurrence of every site from every internal state.

The paper is organized as follows.  Section~\ref{sec:setup} introduces TOM
potentials and RMS speed.  Section~\ref{sec:llt} proves the periodic local
limit theorem, TOM recurrence of \(P_0\), and the Green-function and
potential-kernel asymptotics.  Section~\ref{sec:dichotomy} derives the
one-component drift--dimension dichotomy.  Section~\ref{sec:components}
treats reducible walks.  Section~\ref{sec:traps} studies exact finite traps
and sparse barriers, and Section~\ref{sec:examples} gives examples separating
speed, drift, effective dimension, and recurrence.

\section{OQRWs, TOM potentials, and RMS ballistic speed}
\label{sec:setup}

\subsection{General OQRWs and their dual TOM operators}

Throughout, \(|x|\) denotes the Euclidean norm of a spatial vector in
\(\R^d\) or \(\R^r\), and the modulus of a scalar.  We use \(\norm{\cdot}\)
for Hilbert-space and operator norms, with subscripts when needed.

Let \(\cH\) be a separable Hilbert space and let
\(\{|j\rangle:j\in\Z^d\}\) be the canonical basis of
\(\ell^2(\Z^d)\).  Consider transition operators
\[
  \{B_j^i:i,j\in\Z^d\}\subset B(\cH)
\]
satisfying
\begin{equation}\label{eq:column-normalization}
  \sum_{i\in\Z^d}B_j^{i*}B_j^i=I_{\cH}
  \qquad(j\in\Z^d),
\end{equation}
with strong convergence when the sum is infinite.  On block-diagonal states
\[
  \rho=\sum_{j\in\Z^d}\rho_j\otimes|j\rangle\langle j|,
  \qquad \rho_j\geq0,
  \qquad \sum_j\Tr(\rho_j)=1,
\]
the Schr\"odinger evolution is
\begin{equation}\label{eq:general-oqrw}
  M(\rho)
  =\sum_{i\in\Z^d}
   \left(\sum_{j\in\Z^d}B_j^i\rho_jB_j^{i*}\right)
   \otimes|i\rangle\langle i|.
\end{equation}

The diagonal Heisenberg algebra is
\[
  \algA
  =\ell^\infty\!\left(\Z^d;B(\cH)\right)
  =\left\{\sum_i x_i\otimes|i\rangle\langle i|:
      \sup_i\norm{x_i}<\infty\right\}.
\]
The dual unital completely positive map \(T\colon\algA\to\algA\) is
\begin{equation}\label{eq:dual-TOM}
  T(x)
  =\sum_{j\in\Z^d}
   \left(\sum_{i\in\Z^d}B_j^{i*}x_iB_j^i\right)
   \otimes|j\rangle\langle j|,
  \qquad
  x=\sum_i x_i\otimes|i\rangle\langle i|.
\end{equation}
For a positive \(x\), each inner sum is understood as the monotone
ultraweak limit of its finite partial sums; the formula for general \(x\)
then follows by linearity.  In particular, \(T\) is a normal map.
The duality relation is
\begin{equation}\label{eq:duality}
  \Tr\bigl(M^n(\rho)x\bigr)=\Tr\bigl(\rho T^n(x)\bigr).
\end{equation}

Let
\[
  P_x:=I_{\cH}\otimes|x\rangle\langle x|,
  \qquad
  P_K:=\sum_{x\in K}P_x
\]
for \(x\in\Z^d\) and finite \(K\subset\Z^d\).  If
\(\rho^{(n)}=M^n(\rho^{(0)})\), the measured position law is
\[
  \PP_{\rho^{(0)}}(X_n=x)=\Tr(\rho_x^{(n)}).
\]
For completeness, the associated quantum trajectory starts from
\((j,\rho_j^{(0)}/\Tr(\rho_j^{(0)}))\) with probability
\(\Tr(\rho_j^{(0)})\).  From a position--internal-state pair \((j,\sigma)\),
it moves to
\[
 \left(i,\frac{B_j^i\sigma B_j^{i*}}
                 {\Tr(B_j^i\sigma B_j^{i*})}\right)
 \quad\text{with probability }\Tr(B_j^i\sigma B_j^{i*}).
\]
Zero-probability outcomes are omitted.  The normalization
\eqref{eq:column-normalization} makes this a Markov transition, and its
position marginals agree with the measured law above.

By \eqref{eq:duality},
\begin{equation}\label{eq:occupation-duality}
  \PP_{\rho^{(0)}}(X_n\in K)
  =\Tr\bigl(\rho^{(0)}T^n(P_K)\bigr).
\end{equation}

For \(a\in\algA_+\), define the potential in the extended positive cone by
\[
  U(a):=\sum_{n=0}^\infty T^n(a).
\]
For a normal state, the pairing is understood monotonically:
\[
  \Tr\bigl(\rho U(a)\bigr)
  :=\sum_{n=0}^\infty\Tr\bigl(\rho T^n(a)\bigr)
  \in[0,\infty].
\]
In particular, the Green occupation potential of a finite set is
\begin{equation}\label{eq:green-potential}
  G_{\rho^{(0)}}(K)
  :=\sum_{n=0}^\infty\PP_{\rho^{(0)}}(X_n\in K)
  =\Tr\bigl(\rho^{(0)}U(P_K)\bigr).
\end{equation}

\begin{definition}\label{def:spatial-recurrence}
The walk is \emph{spatially transient from \(\rho^{(0)}\)} if
\[
  G_{\rho^{(0)}}(K)<\infty
\]
for every finite \(K\subset\Z^d\).  It is \emph{recurrent at a site \(x\)
from \(\rho^{(0)}\)} if
\[
  G_{\rho^{(0)}}(\{x\})=\infty.
\]
A nonzero projection \(p\in\algA\) is \emph{recurrent in the TOM sense} if
\[
  \sum_{n=0}^\infty\ip{u}{T^n(p)u}=\infty
  \qquad
  \text{for every }0\neq u\in\Ran(p).
\]
\end{definition}

\begin{remark}\label{rem:recurrence-distinction}
Recurrence at a site from a specified state is a state-dependent occupation
property.  Recurrence of a projection is an operator-theoretic property that
requires divergence on every vector in its range.  Neither notion should be
identified with irreducibility, primitivity, or a spectral gap of the internal
channel.
\end{remark}

\subsection{RMS ballistic speed}

Assume that the initial state has finite position second moment:
\begin{equation}\label{eq:finite-second-moment}
  \E_{\rho^{(0)}}|X_0|^2
  =\sum_{x\in\Z^d}|x|^2\Tr(\rho_x^{(0)})<\infty.
\end{equation}
Define
\begin{align*}
  \overline v_{\mathrm{bal}}(\rho^{(0)})
  &:=\limsup_{n\to\infty}\frac1n
      \bigl(\E_{\rho^{(0)}}|X_n|^2\bigr)^{1/2},\\
  \underline v_{\mathrm{bal}}(\rho^{(0)})
  &:=\liminf_{n\to\infty}\frac1n
      \bigl(\E_{\rho^{(0)}}|X_n|^2\bigr)^{1/2}.
\end{align*}
When these quantities are equal, their common value is denoted by
\(v_{\mathrm{bal}}(\rho^{(0)})\).

For \(N\geq1\), let
\[
  Q_N^2:=\sum_{|x|\leq N}|x|^2P_x\in\algA.
\]

\begin{proposition}\label{prop:elementary-speed}
For every initial state satisfying \eqref{eq:finite-second-moment}, the
following statements hold.

\begin{enumerate}[label=\textup{(\roman*)},leftmargin=2.4em]
\item For every \(n\geq0\),
\begin{equation}\label{eq:truncated-position}
  \E_{\rho^{(0)}}|X_n|^2
  =\lim_{N\to\infty}
    \Tr\bigl(\rho^{(0)}T^n(Q_N^2)\bigr),
\end{equation}
where the limit is monotone.

\item Whenever the first moment is finite,
\begin{equation}\label{eq:drift-speed-bound}
  \limsup_{n\to\infty}\frac{|\E X_n|}{n}
  \leq \overline v_{\mathrm{bal}}(\rho^{(0)}).
\end{equation}

\item If \(\overline v_{\mathrm{bal}}(\rho^{(0)})=0\), then for every
\(v>0\),
\begin{equation}\label{eq:linear-tail-vanishing}
  \PP_{\rho^{(0)}}(|X_n|\geq vn)\longrightarrow0.
\end{equation}

\item If \(\rho^{(0)}=a\rho_1^{(0)}+(1-a)\rho_2^{(0)}\), where
\(0\leq a\leq1\), then
\[
  \E_{\rho^{(0)}}|X_n|^2
  =a\E_{\rho_1^{(0)}}|X_n|^2
   +(1-a)\E_{\rho_2^{(0)}}|X_n|^2,
\]
and consequently
\begin{equation}\label{eq:mixture-speed-upper}
  \overline v_{\mathrm{bal}}(\rho^{(0)})^2
  \leq a\overline v_{\mathrm{bal}}(\rho_1^{(0)})^2
     +(1-a)\overline v_{\mathrm{bal}}(\rho_2^{(0)})^2.
\end{equation}
Here and below, a term with zero weight is omitted if the corresponding
quantity is infinite.
\end{enumerate}
\end{proposition}

\begin{proof}
The block decomposition of \(\rho^{(n)}\) gives
\[
  \Tr(\rho^{(n)}Q_N^2)
  =\sum_{|x|\leq N}|x|^2\Tr(\rho_x^{(n)}).
\]
This increases to \(\E|X_n|^2\), and duality gives
\eqref{eq:truncated-position}.  The bound \eqref{eq:drift-speed-bound} is the
Cauchy--Schwarz inequality
\[
  |\E X_n|\leq \E|X_n|\leq(\E|X_n|^2)^{1/2}.
\]
Equation \eqref{eq:linear-tail-vanishing} follows from Chebyshev's inequality.
Finally, linearity of \(M^n\) makes the second moment affine in the initial
state.  Taking a limsup of the resulting convex combination proves
\eqref{eq:mixture-speed-upper}.
\end{proof}

We say that the OQRW has \textit{range at most \(R\)} if
\begin{equation}\label{eq:finite-range-general}
  B_j^i=0\qquad\text{whenever }|i-j|>R.
\end{equation}
The associated quantum trajectory then has increments of length at most
\(R\), and hence
\begin{equation}\label{eq:pathwise-range-bound}
   |X_n-X_0|\leq Rn
\end{equation}
for every trajectory, including a spatially distributed initial state.
Equivalently, the same bound follows directly by iterating the support relation in \eqref{eq:general-oqrw}.

\begin{proposition}
\label{prop:speed-from-lln}
Assume the walk has finite range and \(\E|X_0|^2<\infty\).  If
\[
  \frac{X_n}{n}\longrightarrow m\in\R^d
\]
in probability, then the RMS ballistic speed exists and
\begin{equation}\label{eq:speed-equals-drift}
  v_{\mathrm{bal}}(\rho^{(0)})=|m|.
\end{equation}
\end{proposition}

\begin{proof}
By \eqref{eq:pathwise-range-bound}, for \(n\geq1\),
\[
  \left|\frac{X_n}{n}\right|^2
  \leq \frac{2|X_0|^2}{n^2}+2R^2
  \leq 2|X_0|^2+2R^2.
\]
Thus \(\{|X_n/n|^2:n\geq1\}\) is uniformly integrable.  Convergence in
probability of \(X_n/n\) to \(m\) therefore implies
\[
  \E\left|\frac{X_n}{n}\right|^2\longrightarrow|m|^2.
\]
By taking square roots we arrive at \eqref{eq:speed-equals-drift}.
\end{proof}

\begin{remark}\label{rem:diffusive-not-ballistic}
The quantity
\[
  \frac1n\E|X_n|^2
\]
is a diffusive second-moment scale, not a ballistic speed.  A centered CLT
together with uniform integrability of \(\{|X_n|^2/n:n\geq1\}\)
implies convergence of this quantity to the trace of the limiting covariance;
it does not imply a positive ballistic speed.  Weak convergence in a CLT
alone does not ensure this second-moment convergence.  Under the hypotheses of Proposition~\ref{prop:speed-from-lln} with
nonzero drift, \(\E|X_n|^2\) has order \(n^2\), so the same quotient diverges linearly.
\end{remark}

\begin{remark}\label{rem:ballistic-velocity-laws}
The distinction between RMS transport and centered ballistic spreading can
be expressed directly in terms of a limiting velocity law.  Suppose that
the position laws satisfy \(X_n/n\Rightarrow\nu\), where \(\nu\) is a
probability measure on \(\R^d\), and that
\(\{|X_n/n|^2:n\geq1\}\) is uniformly integrable.  Write
\(\bar v=\int v\,\nu(dv)\).  Then
\begin{align}
  v_{\mathrm{bal}}^2
    &=\int |v|^2\,\nu(dv),
    \label{eq:rms-velocity-law}\\
  \lim_{n\to\infty}\frac{1}{n^2}
    \E|X_n-\E X_n|^2
    &=\int |v-\bar v|^2\,\nu(dv)
      =\int |v|^2\,\nu(dv)-|\bar v|^2.
    \label{eq:centered-ballistic-coefficient}
\end{align}
Indeed, weak convergence and uniform integrability give convergence of the
first and second moments, and the second identity follows from the variance
decomposition.  For finite-range OQRWs with \(\E|X_0|^2<\infty\), the
required uniform integrability follows from
\eqref{eq:pathwise-range-bound}.  The same argument applies to the
single-time position laws of a finite-range unitary walk started in a
finitely supported state: finite propagation bounds \(|X_n|/n\) uniformly
for \(n\geq1\), without requiring a trajectory with intermediate
measurements.

In the one-component OQRW regime, \(\nu=\delta_m\), so the centered
coefficient in \eqref{eq:centered-ballistic-coefficient} vanishes even when
\(v_{\mathrm{bal}}=|m|>0\).  By contrast, a nondegenerate unitary
velocity law, as in \cite{Konno05,GrimmettJansonScudo04}, has a strictly
positive centered coefficient, including when \(\bar v=0\).
For the finite absorption decomposition of Section~\ref{sec:components},
\(\nu=\sum_\alpha a_\alpha\delta_{m_\alpha}\), and the centered
coefficient equals
\[
  \sum_\alpha a_\alpha|m_\alpha-\bar m|^2,
  \qquad \bar m=\sum_\alpha a_\alpha m_\alpha.
\]
Here a positive coefficient records the mixture of distinct component
drifts.  It does not require coherent interference between dispersion
bands.  None of these moment identities alone determines the sum of local
return probabilities.
\end{remark}

\section{A periodic Fourier local limit theorem and quantitative potentials}
\label{sec:llt}

We now specialize to a homogeneous finite-range OQRW with a
finite-dimensional internal space.  Let \(S\subset\Z^r\) be finite and let
\((L_s)_{s\in S}\subset B(\cH)\), where \(\dim\cH<\infty\), satisfy
\begin{equation}\label{eq:homogeneous-normalization}
  \sum_{s\in S}L_s^*L_s=I_{\cH}.
\end{equation}
The block evolution is
\begin{equation}\label{eq:homogeneous-evolution}
  \rho_x^{(n+1)}
  =\sum_{s\in S}L_s\rho_{x-s}^{(n)}L_s^*,
  \qquad x\in\Z^r.
\end{equation}
Its local channel and Fourier deformations are
\begin{align}
  \cL(\sigma)&:=\sum_{s\in S}L_s\sigma L_s^*,
  \label{eq:local-channel}\\
  \Phi_\theta(\sigma)&:=\sum_{s\in S}e^{i\theta\cdot s}
      L_s\sigma L_s^*,
  \qquad \theta\in\mathbb T^r:=[-\pi,\pi]^r.
  \label{eq:fourier-channel}
\end{align}
We identify opposite faces of \(\mathbb T^r\), so that \(0\) is its distinguished
point.

\begin{lemma}
\label{lem:fourier-contraction}
For every \(\theta\in\mathbb T^r\), the map \(\Phi_\theta\) is a contraction
on the trace class:
\begin{equation}\label{eq:fourier-trace-contraction}
  \norm{\Phi_\theta}_{1\to1}\leq1.
\end{equation}
Moreover,
\begin{equation}\label{eq:fourier-spectral-radius-bound}
  r(\Phi_\theta)\leq1.
\end{equation}
\end{lemma}

\begin{proof}
Let \((e_s)_{s\in S}\) be the canonical basis of \(\ell^2(S)\), and define
\[
  Vv:=\sum_{s\in S}L_sv\otimes e_s,
  \qquad v\in\cH.
\]
The normalization \eqref{eq:homogeneous-normalization} gives
\(V^*V=I_{\cH}\), so \(V\) is an isometry.  Let \(D_\theta\) be the
diagonal unitary on \(\ell^2(S)\) defined by
\[
  D_\theta e_s=e^{i\theta\cdot s}e_s.
\]
With respect to the bilinear trace pairing \(\Tr(\sigma A)\) between the
trace class and \(B(\cH)\), the trace-dual map of \(\Phi_\theta\) is
\[
  \Phi_\theta^*(A)
  =\sum_{s\in S}e^{i\theta\cdot s}L_s^*AL_s
  =V^*(A\otimes D_\theta)V.
\]
Hence
\[
  \norm{\Phi_\theta^*(A)}
  \leq\norm{V}^2\norm{A\otimes D_\theta}
  =\norm{A},
\]
and therefore \(\norm{\Phi_\theta^*}_{\infty\to\infty}\leq1\).
Trace/operator-norm duality now yields
\(\norm{\Phi_\theta}_{1\to1}\leq1\).  The spectral-radius estimate follows
from \(r(\Phi_\theta)\leq\norm{\Phi_\theta}_{1\to1}\).
\end{proof}

\begin{assumption}
\label{ass:fourier}
The following conditions hold.
\begin{enumerate}[label=\textup{(A\arabic*)},leftmargin=2.6em]
\item The local channel \(\cL\) is primitive.  Thus it has a unique faithful
invariant state \(\rho_\infty\), the eigenvalue \(1\) is simple, and no other
eigenvalue lies on the unit circle.

\item The Fourier peripheral set
\begin{equation}\label{eq:peripheral-set}
  \Theta:=\{\xi\in\mathbb T^r:r(\Phi_\xi)=1\}
\end{equation}
is finite.  For every \(\xi\in\Theta\), the map \(\Phi_\xi\) has a unique
eigenvalue \(\zeta_\xi\) on the unit circle, and this eigenvalue is
algebraically simple.  Moreover,
\[
  \zeta_0=1,\qquad
  \zeta_\xi\neq1\quad(\xi\neq0),
\]
and every \(\zeta_\xi\) is a root of unity.

\item Let \(\lambda_\xi(\xi+\eta)\) be the analytic eigenvalue branch issuing
from \(\zeta_\xi\).  For real \(\eta\) near \(0\), these branches have the
common expansion
\begin{equation}\label{eq:lambda-expansion}
  \log\!\left(\zeta_\xi^{-1}\lambda_\xi(\xi+\eta)\right)
  =i\,m\cdot\eta-\frac12\eta^T\Sigma\eta+O(|\eta|^3),
  \qquad \xi\in\Theta,
\end{equation}
where \(m\in\R^r\) and \(\Sigma\) is a symmetric positive-definite real
\(r\times r\) matrix.
\end{enumerate}
\end{assumption}

Condition~\textup{(A1)} permits analytic perturbation of the principal
eigenvalue \cite{Kato95}; the required spectral facts for positive maps are
standard \cite{EvansHoeghKrohn78}.  Conditions~\textup{(A2)}--\textup{(A3)}
retain finitely many periodic Fourier modes.  Strong aperiodicity is the
special case \(\Theta=\{0\}\).  For scalar walks whose active jumps generate the lattice,
\(\Phi_\xi=1\) forces \(\xi=0\).  For an OQRW, however, the absence of an
eigenvalue \(1\) at a nonzero Fourier point is an additional assumption;
passing to the lattice generated by the active jumps does not ensure it.
Example~\ref{ex:primitive-zero-speed-transient} exhibits this distinction.
This part of \textup{(A2)} is used both for the mean-one return mask and for
the absence of further resolvent singularities.  The positive-definite quadratic term in
\textup{(A3)} is the standard nondegeneracy condition for a
Markov-additive local limit theorem.

Figure~\ref{fig:oqrw-proof-architecture} summarizes the two analytic routes
from the peripheral Fourier expansion to the recurrence and potential
results proved below.

\FloatBarrier
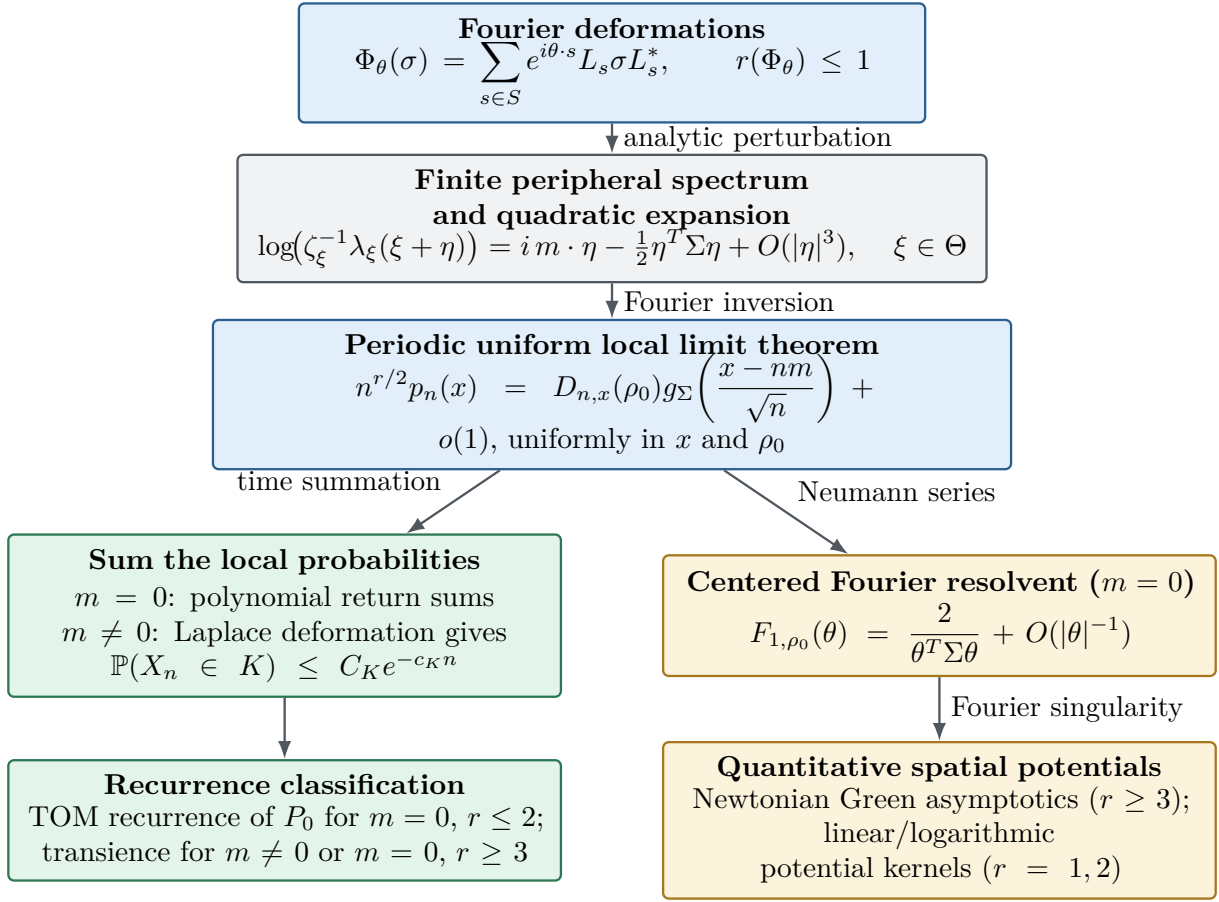
\begin{figure}[t]
\centering
\resizebox{\textwidth}{!}{%
\begin{tikzpicture}
  \node[oqrw blue,text width=7.3cm] (fourier) at (0,0) {
    \textbf{Fourier deformations}\\[-1pt]
    $\displaystyle \Phi_\theta(\sigma)
       =\sum_{s\in S}e^{i\theta\cdot s}L_s\sigma L_s^*$,
    \qquad $r(\Phi_\theta)\leq1$
  };

  \node[oqrw gray,text width=8.8cm] (spectral) at (0,-1.9) {
    \textbf{Finite peripheral spectrum and quadratic expansion}\\[-1pt]
    $\displaystyle
      \log\!\bigl(\zeta_\xi^{-1}\lambda_\xi(\xi+\eta)\bigr)
      =i\,m\cdot\eta-\tfrac12\eta^T\Sigma\eta+O(|\eta|^3)$,
    \quad $\xi\in\Theta$
  };

  \node[oqrw blue,text width=9.4cm] (llt) at (0,-4.05) {
    \textbf{Periodic uniform local limit theorem}\\[-1pt]
    $\displaystyle n^{r/2}p_n(x)
      =D_{n,x}(\rho_0)
       g_\Sigma\!\left(\frac{x-nm}{\sqrt n}\right)+o(1)$,
    uniformly in $x$ and $\rho_0$
  };

  \node[oqrw green,text width=6.4cm] (returns) at (-4.0,-6.75) {
    \textbf{Sum the local probabilities}\\[1pt]
    $m=0$: polynomial return sums\\
    $m\neq0$: Laplace deformation gives\\
    $\mathbb P(X_n\in K)\leq C_Ke^{-c_Kn}$
  };

  \node[oqrw gold,text width=6.4cm] (resolvent) at (4.0,-6.75) {
    \textbf{Centered Fourier resolvent ($m=0$)}\\[1pt]
    $\displaystyle F_{1,\rho_0}(\theta)
      =\frac{2}{\theta^T\Sigma\theta}
       +O(|\theta|^{-1})$
  };

  \node[oqrw green,text width=6.4cm] (recurrence) at (-4.0,-9.25) {
    \textbf{Recurrence classification}\\[-1pt]
    TOM recurrence of $P_0$ for $m=0$, $r\leq2$;\\
    transience for $m\neq0$ or $m=0$, $r\geq3$
  };

  \node[oqrw gold,text width=6.4cm] (potentials) at (4.0,-9.25) {
    \textbf{Quantitative spatial potentials}\\[-1pt]
    Newtonian Green asymptotics $(r\geq3)$;\\
    linear/logarithmic\\potential kernels $(r=1,2)$
  };

  \draw[oqrw arrow] (fourier) -- node[right,font=\small] {analytic perturbation} (spectral);
  \draw[oqrw arrow] (spectral) -- node[right,font=\small] {Fourier inversion} (llt);
  \draw[oqrw arrow] (llt) -- node[above left,font=\small] {time summation} (returns);
  \draw[oqrw arrow] (llt) -- node[above right,font=\small] {Neumann series} (resolvent);
  \draw[oqrw arrow] (returns) -- (recurrence);
  \draw[oqrw arrow] (resolvent) -- node[right,font=\small] {Fourier singularity} (potentials);
\end{tikzpicture}%
}
\caption{Local-to-global architecture of the main argument.  The periodic
local theorem controls time returns, while the centered resolvent singularity
controls the spatial profile of Green functions and potential kernels.}
\label{fig:oqrw-proof-architecture}
\end{figure}
\FloatBarrier

For a localized initial state
\begin{equation}\label{eq:localized-initial-state}
  \rho^{(0)}=\rho_0\otimes|0\rangle\langle0|,
  \qquad \rho_0\geq0,
  \qquad \Tr(\rho_0)=1,
\end{equation}
write
\[
  p_n(x):=\PP_{\rho^{(0)}}(X_n=x).
\]
Let
\begin{equation}\label{eq:gaussian-density}
  g_\Sigma(y)
  :=\frac{1}{(2\pi)^{r/2}\sqrt{\det\Sigma}}
    \exp\!\left(-\frac12y^T\Sigma^{-1}y\right).
\end{equation}

Let \(q\) be a common order of the roots
\((\zeta_\xi)_{\xi\in\Theta}\), and let \(\Pi_\xi\) be the spectral
projection of \(\Phi_\xi\) at \(\zeta_\xi\).  For a density matrix \(\rho_0\)
define
\begin{align}
  a_\xi(\rho_0)&:=\Tr\bigl(\Pi_\xi(\rho_0)\bigr),
  \label{eq:periodic-amplitude}\\
  D_{n,x}(\rho_0)
  &:=\sum_{\xi\in\Theta}
     \zeta_\xi^n e^{-i\xi\cdot x}a_\xi(\rho_0).
  \label{eq:periodic-mask}
\end{align}
Because \(\Pi_0(\sigma)=\Tr(\sigma)\rho_\infty\), one has
\begin{equation}\label{eq:zero-amplitude}
  a_0(\rho_0)=1.
\end{equation}

\begin{theorem}
\label{thm:uniform-llt}
Under Assumption~\ref{ass:fourier}, the following statements hold:
\begin{enumerate}[label=\textup{(\roman*)},leftmargin=2.4em]
\item The position law satisfies the periodic uniform local limit theorem
\begin{equation}\label{eq:uniform-llt}
  \sup_{x\in\Z^r}
  \left|
    n^{r/2}p_n(x)
    -D_{n,x}(\rho_0)
     g_\Sigma\!\left(\frac{x-nm}{\sqrt n}\right)
  \right|
  \longrightarrow0.
\end{equation}

\item If \(m=0\), then for every finite \(K\subset\Z^r\) there is
\(C_K<\infty\) such that
\begin{equation}\label{eq:centered-upper-return}
  \PP_{\rho^{(0)}}(X_n\in K)\leq C_Kn^{-r/2}
  \qquad(n\geq1).
\end{equation}
For \(0\leq\ell<q\), put
\begin{equation}\label{eq:residue-mask}
  D_\ell(\rho_0)
  :=\sum_{\xi\in\Theta}\zeta_\xi^\ell a_\xi(\rho_0).
\end{equation}
Then
\begin{equation}\label{eq:residue-return-limit}
  n^{r/2}p_n(0)\longrightarrow
  g_\Sigma(0)D_\ell(\rho_0)
  \quad\text{as }n\to\infty,\quad n\equiv\ell\pmod q,
\end{equation}
where
\begin{equation}\label{eq:mask-mean}
  D_\ell(\rho_0)\geq0,
  \qquad
  \frac1q\sum_{\ell=0}^{q-1}D_\ell(\rho_0)=1.
\end{equation}
In particular,
\begin{align}
  \sum_{n=1}^N p_n(0)
  &=2g_\Sigma(0)N^{1/2}+o(N^{1/2}),
  &&r=1,
  \label{eq:cumulative-return-r1}\\
  \sum_{n=1}^N p_n(0)
  &=g_\Sigma(0)\log N+o(\log N),
  &&r=2.
  \label{eq:cumulative-return-r2}
\end{align}

\item If \(m\neq0\), then for every finite \(K\subset\Z^r\) there are
\(C_K,c_K>0\) such that
\begin{equation}\label{eq:drift-exponential-return}
  \PP_{\rho^{(0)}}(X_n\in K)\leq C_Ke^{-c_Kn}
  \qquad(n\geq0).
\end{equation}
\end{enumerate}
The convergence and the estimates may be chosen uniformly over all density
matrices \(\rho_0\).
\end{theorem}

\begin{proof}
We prove the assertions with constants uniform in the initial density
matrix.  Since \(B(\cH)\) is finite dimensional, all norms on it and on the
space of superoperators are equivalent.  We fix one such norm throughout.
For a density matrix \(\rho\), every linear functional of the form
\(\rho\mapsto\Tr(A\rho)\) is bounded uniformly in \(\rho\); this observation
will be used repeatedly.

\smallskip
\noindent\emph{Step 1: Fourier inversion.}
Introduce the matrix-valued Fourier transform
\[
  \widehat\rho_n(\theta)
  :=\sum_{x\in\Z^r}e^{i\theta\cdot x}\rho_x^{(n)}.
\]
The sum is absolutely convergent in trace norm because
\(\sum_x\Tr(\rho_x^{(n)})=1\).  Using
\eqref{eq:homogeneous-evolution}, the finiteness of \(S\), and the change of
variables \(y=x-s\), we obtain
\begin{align*}
  \widehat\rho_{n+1}(\theta)
  &=\sum_{x\in\Z^r}\sum_{s\in S}
    e^{i\theta\cdot x}L_s\rho_{x-s}^{(n)}L_s^*\\
  &=\sum_{s\in S}e^{i\theta\cdot s}
    L_s\widehat\rho_n(\theta)L_s^*
   =\Phi_\theta\bigl(\widehat\rho_n(\theta)\bigr).
\end{align*}
Since the walk starts at the origin,
\(\widehat\rho_0(\theta)=\rho_0\), and induction yields
\[
  \widehat\rho_n(\theta)=\Phi_\theta^n(\rho_0).
\]
Taking traces proves
\begin{equation}\label{eq:characteristic-fourier}
  \E_{\rho^{(0)}}e^{i\theta\cdot X_n}
  =\Tr\bigl(\Phi_\theta^n(\rho_0)\bigr).
\end{equation}
The function on the right is a trigonometric polynomial, and its Fourier
coefficient at \(x\) is \(p_n(x)\).  Fourier inversion on \(\mathbb T^r\)
therefore gives
\begin{equation}\label{eq:fourier-inversion}
  p_n(x)
  =\frac{1}{(2\pi)^r}\int_{\mathbb T^r}
    e^{-i\theta\cdot x}\Tr\bigl(\Phi_\theta^n(\rho_0)\bigr)\,d\theta.
\end{equation}

\smallskip
\noindent\emph{Step 2: spectral estimates.}
Because \(\Theta\) is finite, we may choose \(\delta>0\) so that the torus
neighborhoods
\[
  V_\xi:=\{\xi+\eta:|\eta|<\delta\},
  \qquad \xi\in\Theta,
\]
are pairwise disjoint.  Fix \(\xi\in\Theta\).  The eigenvalue
\(\zeta_\xi\) is algebraically simple and isolated from the remaining
spectrum of \(\Phi_\xi\).  Hence a positively oriented circle
\(\mathcal C_\xi\) can be chosen around \(\zeta_\xi\) containing no other
spectral point.  After decreasing \(\delta\), the Riesz projection
\[
  \Pi_{\xi,\eta}
  :=\frac{1}{2\pi i}\int_{\mathcal C_\xi}
    (zI-\Phi_{\xi+\eta})^{-1}\,dz
\]
is well defined and analytic for \(|\eta|<\delta\).  It has rank one and
commutes with \(\Phi_{\xi+\eta}\).  Define
\[
  N_{\xi,\eta}
  :=\Phi_{\xi+\eta}(I-\Pi_{\xi,\eta}).
\]
Then \(\Pi_{\xi,\eta}N_{\xi,\eta}
=N_{\xi,\eta}\Pi_{\xi,\eta}=0\), and, for every integer \(n\geq0\),
\begin{equation}\label{eq:spectral-decomposition}
  \Phi_{\xi+\eta}^n
  =\lambda_\xi(\xi+\eta)^n\Pi_{\xi,\eta}
   +N_{\xi,\eta}^n(I-\Pi_{\xi,\eta}).
\end{equation}
For \(n\geq1\), the last term equals \(N_{\xi,\eta}^n\).

At \(\eta=0\), the spectral radius of \(N_{\xi,0}\) is strictly smaller
than one.  Continuity of the spectrum, compactness of the closed
\(\eta\)-balls, and finiteness of \(\Theta\) give numbers
\(\kappa_0<\kappa<1\) such that
\[
  r(N_{\xi,\eta})\leq\kappa_0
  \qquad(\xi\in\Theta,\ |\eta|\leq\delta).
\]
The circle \(|z|=\kappa\) belongs to the common resolvent set.  Therefore
\[
  N_{\xi,\eta}^n
  =\frac{1}{2\pi i}\int_{|z|=\kappa}
    z^n(zI-N_{\xi,\eta})^{-1}\,dz.
\]
The resolvents are uniformly bounded on the compact parameter set, and hence
\begin{equation}\label{eq:local-remainder-bound}
  \sup_{\xi\in\Theta,\ |\eta|<\delta}
  \norm{N_{\xi,\eta}^n}\leq C\kappa^n
\end{equation}
with \(C\) independent of \(\xi,\eta,n\).

Put
\begin{equation}\label{eq:amplitude}
  a_\xi(\eta;\rho_0)
  :=\Tr\bigl(\Pi_{\xi,\eta}(\rho_0)\bigr).
\end{equation}
This is analytic in \(\eta\), linear in \(\rho_0\), and
\(a_\xi(0;\rho_0)=a_\xi(\rho_0)\).  Uniform bounds on the finitely many
analytic projections and their first derivatives imply
\begin{equation}\label{eq:amplitude-uniform-bounds}
  \sup_{\rho_0}\abs{a_\xi(\eta;\rho_0)}\leq C,
  \qquad
  \sup_{\rho_0}
  \abs{a_\xi(\eta;\rho_0)-a_\xi(\rho_0)}
  \leq C|\eta|,
\end{equation}
uniformly in \(\xi\in\Theta\).  The suprema here and below are taken over
all density matrices.

Let \(\sigma_*>0\) be the smallest eigenvalue of \(\Sigma\).  Since
\(\Theta\) is finite, the remainder in \eqref{eq:lambda-expansion} is
bounded by \(M|\eta|^3\), with \(M\) independent of \(\xi\).  Consequently,
after reducing \(\delta\),
\[
  \log\abs{\lambda_\xi(\xi+\eta)}
  =-\frac12\eta^T\Sigma\eta+O(|\eta|^3)
  \leq-\frac{\sigma_*}{4}|\eta|^2.
\]
Thus, for \(c=\sigma_*/4\),
\begin{equation}\label{eq:lambda-gaussian-majorant}
  |\lambda_\xi(\xi+\eta)|\leq e^{-c|\eta|^2}
  \qquad(\xi\in\Theta,\ |\eta|<\delta)
\end{equation}
for all the peripheral neighborhoods.

Finally, set
\[
  W:=\mathbb T^r\setminus\bigcup_{\xi\in\Theta}V_\xi.
\]
By \eqref{eq:peripheral-set}, \(r(\Phi_\theta)<1\) on \(W\).  Continuity of
the spectral radius and compactness give
\(\sup_{\theta\in W}r(\Phi_\theta)<1\).  Enlarge the previously chosen
\(\kappa<1\), if necessary, so that it also exceeds this supremum; the
bound \eqref{eq:local-remainder-bound} remains valid.  Using the resolvent
circle \(|z|=\kappa\) then yields
\begin{equation}\label{eq:away-origin-bound}
  \sup_{\theta\in W}
  \norm{\Phi_\theta^n}\leq C\kappa^n.
\end{equation}

\smallskip
\noindent\emph{Step 3: the periodic uniform local limit.}
Insert \eqref{eq:spectral-decomposition} into
\eqref{eq:fourier-inversion} and split the torus into \(W\) and the
\(V_\xi\).  By \eqref{eq:local-remainder-bound} and
\eqref{eq:away-origin-bound}, the sum of all complementary contributions
has modulus at most \(C\kappa^n\), uniformly in \(x\) and \(\rho_0\).
Indeed, the factors \(e^{-i\theta\cdot x}\) have modulus one and the trace
functional is bounded.  After multiplication by \(n^{r/2}\), this error is
\[
  O(n^{r/2}\kappa^n)=o(1).
\]

Consider the dominant contribution from \(V_\xi\).  Writing
\(\theta=\xi+\eta\), multiplying by \(n^{r/2}\), and setting
\(u=\sqrt n\,\eta\), it becomes
\begin{align}
 I_{n,\xi}(x;\rho_0)
 &:=
 \frac{\zeta_\xi^ne^{-i\xi\cdot x}}{(2\pi)^r}
 \int_{|u|<\delta\sqrt n}
 e^{-iu\cdot x/\sqrt n}
 \left(\zeta_\xi^{-1}
   \lambda_\xi(\xi+u/\sqrt n)\right)^n
 a_\xi(u/\sqrt n;\rho_0)\,du.
 \label{eq:scaled-peripheral-contribution}
\end{align}
Set
\[
  y_{n,x}:=\frac{x-nm}{\sqrt n}
\]
and define a function on \(\R^r\) by
\begin{align*}
  H_{n,\xi}(u;\rho_0)
  &:=
  \mathbf1_{\{|u|<\delta\sqrt n\}}
  e^{-i\sqrt n\,m\cdot u}
  \left(\zeta_\xi^{-1}
    \lambda_\xi(\xi+u/\sqrt n)\right)^n
  a_\xi(u/\sqrt n;\rho_0).
\end{align*}
The identity
\[
  e^{-iu\cdot x/\sqrt n}
  =e^{-iu\cdot y_{n,x}}e^{-i\sqrt n\,m\cdot u}
\]
shows that \eqref{eq:scaled-peripheral-contribution} equals
\begin{equation}\label{eq:scaled-contribution-H}
  \zeta_\xi^ne^{-i\xi\cdot x}
  \frac{1}{(2\pi)^r}\int_{\R^r}
    e^{-iu\cdot y_{n,x}}H_{n,\xi}(u;\rho_0)\,du.
\end{equation}

For fixed \(u\), \eqref{eq:lambda-expansion} gives
\begin{align*}
 &n\log\!\left(
   \zeta_\xi^{-1}\lambda_\xi(\xi+u/\sqrt n)\right)
   -i\sqrt n\,m\cdot u\\
 &\qquad
 =-\frac12u^T\Sigma u+O\!\left(\frac{|u|^3}{\sqrt n}\right).
\end{align*}
Together with \eqref{eq:amplitude-uniform-bounds}, this implies
\[
  H_{n,\xi}(u;\rho_0)
  \longrightarrow
  a_\xi(\rho_0)e^{-\frac12u^T\Sigma u}
\]
for each fixed \(u\), uniformly over density matrices \(\rho_0\).
Moreover, \eqref{eq:lambda-gaussian-majorant} yields
\[
  \abs{H_{n,\xi}(u;\rho_0)}
  \leq C e^{-c|u|^2}.
\]
The right-hand side is integrable and independent of \(n,\xi,\rho_0\).
The uniform amplitude estimate in
\eqref{eq:amplitude-uniform-bounds}, followed by dominated convergence,
therefore gives the stronger statement
\begin{equation}\label{eq:uniform-L1-convergence}
 \sup_{\rho_0}
 \norm{
 H_{n,\xi}(\,\cdot\,;\rho_0)
 -a_\xi(\rho_0)e^{-\frac12(\,\cdot\,)^T
   \Sigma(\,\cdot\,)}
 }_{L^1(\R^r)}
 \longrightarrow0.
\end{equation}
For completeness, the uniformity follows by splitting the difference into
the eigenvalue part and the amplitude part.  The first is independent of
\(\rho_0\) up to a uniformly bounded factor; the second is bounded by
\(C\min\{|u|/\sqrt n,1\}e^{-c|u|^2}\).  Both terms converge to zero in
\(L^1\).

For any \(F,G\in L^1(\R^r)\),
\[
  \sup_{y\in\R^r}
  \left|\frac{1}{(2\pi)^r}
  \int_{\R^r}e^{-iu\cdot y}(F(u)-G(u))\,du\right|
  \leq\frac{\norm{F-G}_{L^1}}{(2\pi)^r}.
\]
Applying this inequality to \eqref{eq:uniform-L1-convergence} and using
\[
  \frac{1}{(2\pi)^r}\int_{\R^r}
  e^{-iu\cdot y}e^{-\frac12u^T\Sigma u}\,du
  =g_\Sigma(y),
\]
we obtain, uniformly in \(x\) and \(\rho_0\),
\[
  I_{n,\xi}(x;\rho_0)
  =\zeta_\xi^ne^{-i\xi\cdot x}
   a_\xi(\rho_0)g_\Sigma(y_{n,x})+o(1).
\]
Summing this identity over the finite set \(\Theta\), and adding the
exponentially small complementary contribution, gives exactly
\eqref{eq:uniform-llt}.

\smallskip
\noindent\emph{Step 4: centered return estimates.}
Assume \(m=0\).  The mask is uniformly bounded:
\[
  \sup_{n,x,\rho_0}\abs{D_{n,x}(\rho_0)}
  \leq\sum_{\xi\in\Theta}
       \sup_{\rho_0}\abs{a_\xi(\rho_0)}
  <\infty.
\]
Since \(g_\Sigma\) is bounded, \eqref{eq:uniform-llt} implies, for all
sufficiently large \(n\),
\[
  \sup_{\rho_0}\sup_{x\in\Z^r}p_n(x)\leq Cn^{-r/2}.
\]
Increasing \(C\) covers the finitely many smaller values of \(n\).  Hence,
for finite \(K\),
\[
  \PP_{\rho^{(0)}}(X_n\in K)
  =\sum_{x\in K}p_n(x)
  \leq |K|Cn^{-r/2},
\]
which proves \eqref{eq:centered-upper-return}, uniformly in \(\rho_0\).

At \(x=0\), the periodic mask depends only on \(n\bmod q\):
\[
  D_{n,0}(\rho_0)=D_\ell(\rho_0)
  \qquad\text{when }n\equiv\ell\pmod q.
\]
Since \(m=0\), \(g_\Sigma((0-nm)/\sqrt n)=g_\Sigma(0)\).
Thus \eqref{eq:uniform-llt}, restricted to the residue class
\(n\equiv\ell\pmod q\), gives
\[
  n^{r/2}p_n(0)
  =g_\Sigma(0)D_\ell(\rho_0)+o(1),
\]
which is \eqref{eq:residue-return-limit}.  The left-hand side is
nonnegative and \(g_\Sigma(0)>0\); consequently
\(D_\ell(\rho_0)\geq0\).

Next, \(\zeta_0=1\), \(a_0(\rho_0)=1\), and, for
\(\xi\neq0\), \(\zeta_\xi\) is a nontrivial \(q\)-th root of unity.
Therefore
\begin{align*}
  \frac1q\sum_{\ell=0}^{q-1}D_\ell(\rho_0)
  &=\sum_{\xi\in\Theta}a_\xi(\rho_0)
    \left(\frac1q\sum_{\ell=0}^{q-1}\zeta_\xi^\ell\right)\\
  &=a_0(\rho_0)=1.
\end{align*}
This proves \eqref{eq:mask-mean}.

We now justify the two cumulative asymptotics.  Write
\[
  D_{n,0}(\rho_0)=1+b_n(\rho_0).
\]
The sequence \(b_n(\rho_0)\) is \(q\)-periodic and has mean zero over a
period.  Consequently, its partial sums
\[
  B_N(\rho_0):=\sum_{n=1}^Nb_n(\rho_0)
\]
are uniformly bounded in \(N\) and in \(\rho_0\).  Abel summation gives, for
each decreasing sequence \(w_n>0\),
\[
  \sum_{n=1}^Nb_n(\rho_0)w_n
  =B_N(\rho_0)w_N
   +\sum_{n=1}^{N-1}B_n(\rho_0)(w_n-w_{n+1}).
\]
Taking \(w_n=n^{-1/2}\) and \(w_n=n^{-1}\), respectively, shows
\begin{align}
  \sum_{n=1}^ND_{n,0}(\rho_0)n^{-1/2}
  &=\sum_{n=1}^Nn^{-1/2}+O(1)
   =2N^{1/2}+O(1),
  \label{eq:periodic-sum-r1}\\
  \sum_{n=1}^ND_{n,0}(\rho_0)n^{-1}
  &=\sum_{n=1}^Nn^{-1}+O(1)
   =\log N+O(1),
  \label{eq:periodic-sum-r2}
\end{align}
uniformly in \(\rho_0\).

Finally, let
\[
  \varepsilon_n(\rho_0)
  :=n^{r/2}p_n(0)-g_\Sigma(0)D_{n,0}(\rho_0).
\]
The uniform local theorem says
\(\sup_{\rho_0}|\varepsilon_n(\rho_0)|\to0\).  If
\(A_N^{(\alpha)}:=\sum_{n=1}^Nn^{-\alpha}\) with
\(\alpha\in\{1/2,1\}\), then
\[
  \sup_{\rho_0}
  \abs{\sum_{n=1}^N\varepsilon_n(\rho_0)n^{-\alpha}}
  =o\!\left(A_N^{(\alpha)}\right).
\]
Indeed, after choosing \(n_0\) so that
\(\sup_{\rho_0}|\varepsilon_n(\rho_0)|\leq\epsilon\) for \(n\geq n_0\),
the initial segment is \(O(1)\) and the remaining sum is at most
\(\epsilon A_N^{(\alpha)}\).  Combining this estimate with
\eqref{eq:periodic-sum-r1}--\eqref{eq:periodic-sum-r2} proves
\eqref{eq:cumulative-return-r1}--\eqref{eq:cumulative-return-r2}.

\smallskip
\noindent\emph{Step 5: a real Laplace deformation.}
In the setting of \cite{CarbonePautrat15}, the exponential-return conclusion
also follows from its concentration estimate.  The following argument gives constants
uniform over all initial density matrices.
Suppose \(m\neq0\), and put \(e=m/|m|\).  For
\(t\in\R^r\) near \(0\), set
\begin{equation}\label{eq:laplace-deformation}
  \Psi_t(\sigma):=\sum_{s\in S}e^{t\cdot s}L_s\sigma L_s^*.
\end{equation}
The same path expansion as in Step~1, now with positive
weights instead of characters, gives
\begin{equation}\label{eq:laplace-transform-identity}
  \E_{\rho^{(0)}}e^{t\cdot X_n}
  =\Tr\bigl(\Psi_t^n(\rho_0)\bigr).
\end{equation}
Indeed, each path with total displacement \(X_n\) acquires the product
\(\prod_{k=1}^ne^{t\cdot s_k}=e^{t\cdot X_n}\).

The map \(\Psi_t\) is completely positive, with Kraus operators
\((e^{t\cdot s/2}L_s)_{s\in S}\).  Since these scalar factors are strictly
positive, the Kraus words that witness primitivity are exactly the same as
for \(\cL\).  Hence \(\Psi_t\) is primitive for all real \(t\) sufficiently
close to \(0\).  Perron--Frobenius theory for primitive positive maps gives
a positive algebraically simple spectral radius \(\mu(t)\), analytic near
\(0\), with \(\mu(0)=1\); the remainder of the spectrum stays uniformly
separated from \(\mu(t)\).

Let \(\rho(t)\) be the corresponding analytic right eigenvector, normalized
by \(\Tr(\rho(t))=1\):
\[
  \Psi_t(\rho(t))=\mu(t)\rho(t),
  \qquad \rho(0)=\rho_\infty.
\]
Differentiating in the \(j\)-th coordinate at \(t=0\) gives
\[
  (\partial_j\Psi_0)(\rho_\infty)
  +\cL(\partial_j\rho(0))
  =\partial_j\mu(0)\rho_\infty+\partial_j\rho(0).
\]
Taking traces, using trace preservation of \(\cL\) and
\(\Tr(\partial_j\rho(0))=0\), yields
\[
  \partial_j\mu(0)
  =\Tr\bigl((\partial_j\Psi_0)(\rho_\infty)\bigr)
  =\Tr\!\left(\sum_{s\in S}s_j
    L_s\rho_\infty L_s^*\right).
\]
On the other hand, \(\Psi_t=\Phi_{-it}\); analytic continuation of
\eqref{eq:lambda-expansion} at \(\xi=0\) identifies the last expression
with \(m_j\).  Thus, for
\(\Lambda(t):=\log\mu(t)\),
\begin{equation}\label{eq:laplace-gradient}
  \nabla\Lambda(0)=m.
\end{equation}
In fact, the same continuation gives the second-order expansion
\begin{equation}\label{eq:laplace-expansion}
  \Lambda(t)
  =m\cdot t+\frac12t^T\Sigma t+O(|t|^3).
\end{equation}
Taking \(t=-\gamma e\), we obtain, for sufficiently small \(\gamma>0\),
\[
  \Lambda(-\gamma e)
  =-\gamma|m|+\frac{\gamma^2}{2}e^T\Sigma e+O(\gamma^3)<0.
\]
Thus \(\mu(-\gamma e)<1\).  By choosing \(\gamma\) still smaller if
necessary, the complementary spectrum also lies in a disk of radius
strictly smaller than one.  Fix \(\kappa\in(0,1)\) larger than the modulus
of every eigenvalue of \(\Psi_{-\gamma e}\).  The resolvent formula for
powers gives
\[
  \norm{\Psi_{-\gamma e}^n}\leq C\kappa^n.
\]
Using positivity, \eqref{eq:laplace-transform-identity}, and boundedness of
the trace functional, we conclude, uniformly over density matrices
\(\rho_0\), that
\begin{equation}\label{eq:mgf-decay}
  \E_{\rho^{(0)}}e^{-\gamma e\cdot X_n}
  =\Tr\bigl(\Psi_{-\gamma e}^n(\rho_0)\bigr)
  \leq C\kappa^n=Ce^{-cn},
  \qquad c:=-\log\kappa>0.
\end{equation}

The assertion is trivial for \(K=\varnothing\).  For a nonempty finite set
\(K\), let
\[
  R_K:=\max_{x\in K}e\cdot x.
\]
On the event \(\{e\cdot X_n\leq R_K\}\),
\(e^{-\gamma e\cdot X_n}\geq e^{-\gamma R_K}\).  Markov's inequality and
\eqref{eq:mgf-decay} therefore give
\[
  \PP_{\rho^{(0)}}(X_n\in K)
  \leq\PP_{\rho^{(0)}}(e\cdot X_n\leq R_K)
  \leq e^{\gamma R_K}\E_{\rho^{(0)}}e^{-\gamma e\cdot X_n}
  \leq Ce^{\gamma R_K}e^{-cn}.
\]
Taking \(C_K:=Ce^{\gamma R_K}\) and \(c_K:=c\), and increasing \(C_K\) if
needed to cover finitely many small \(n\), proves
\eqref{eq:drift-exponential-return}.  Every constant used in this step is
independent of \(\rho_0\), completing the proof of the asserted uniformity.
\end{proof}

\begin{corollary}
\label{cor:one-component-return}
Under Assumption~\ref{ass:fourier}, the walk is spatially transient when
\(m\neq0\).  If \(m=0\), it is spatially transient for \(r\geq3\) and
recurrent at the origin for \(r\in\{1,2\}\).
\end{corollary}

\begin{proof}
For nonzero drift, sum \eqref{eq:drift-exponential-return}.  In the centered
case, \eqref{eq:centered-upper-return} is summable exactly when \(r/2>1\),
while \eqref{eq:cumulative-return-r1}--\eqref{eq:cumulative-return-r2}
give divergence when \(r\in\{1,2\}\).
\end{proof}

\begin{corollary}
\label{cor:tom-origin-recurrence}
Under Assumption~\ref{ass:fourier}, if \(m=0\) and
\(r\in\{1,2\}\), then \(P_0\) is recurrent in the TOM sense.
\end{corollary}

\begin{proof}
Every nonzero \(u\in\Ran(P_0)\) has the form
\(u=\eta\otimes|0\rangle\) with \(0\neq\eta\in\cH\).  Set
\[
  \rho_\eta:=\frac{|\eta\rangle\langle\eta|}{\norm{\eta}^2}.
\]
By the duality formula \eqref{eq:occupation-duality},
\[
  \ip{u}{T^n(P_0)u}
  =\norm{\eta}^2
    \PP_{\rho_\eta\otimes|0\rangle\langle0|}(X_n=0).
\]
The cumulative asymptotics
\eqref{eq:cumulative-return-r1}--\eqref{eq:cumulative-return-r2} apply to
every density matrix \(\rho_\eta\), and their leading constants are strictly
positive.  Thus the sum over \(n\) diverges for every nonzero
\(u\in\Ran(P_0)\), which is precisely TOM recurrence of \(P_0\).
\end{proof}

\subsection{Green functions and potential kernels}

The periodic local theorem decides summability.  The resolvent contains
strictly more information: its singularity at the stationary Fourier point
determines the spatial profile of the potential.  For the localized state
\eqref{eq:localized-initial-state}, write
\begin{equation}\label{eq:point-green-function}
  G_{\rho_0}(x):=\sum_{n=0}^\infty p_n(x).
\end{equation}
In the recurrent dimensions define, whenever the series converges, the
potential kernel
\begin{equation}\label{eq:potential-kernel}
  \mathfrak a_{\rho_0}(x)
  :=\sum_{n=0}^\infty\bigl(p_n(0)-p_n(x)\bigr).
\end{equation}
The series in \eqref{eq:potential-kernel} need not be absolutely convergent.

\begin{lemma}
\label{lem:fourier-singularity}
Let \(\Sigma\) be positive definite and set
\[
  R_\Sigma(x):=(x^T\Sigma^{-1}x)^{1/2}.
\]
Let \(\chi\) be a real-valued smooth even cutoff supported in a sufficiently small
neighborhood of \(0\) and equal to \(1\) near \(0\).  Then, as
\(R_\Sigma(x)\to\infty\),
\begin{align}
 \frac{1}{(2\pi)^r}\int_{\R^r}
 e^{-i\theta\cdot x}\frac{2\chi(\theta)}
 {\theta^T\Sigma\theta}\,d\theta
 &=
 \frac{\Gamma(\frac r2-1)}
 {2\pi^{r/2}\sqrt{\det\Sigma}}\,
 R_\Sigma(x)^{2-r}
 +O\!\left(R_\Sigma(x)^{1-r}\right),
 &&r\geq3,
 \label{eq:riesz-asymptotic}\\
 \frac{1}{2\pi}\operatorname{pv}\!\int_{\R}
 (1-e^{-i\theta x})\frac{2\chi(\theta)}
 {\Sigma\theta^2}\,d\theta
 &=\frac{|x|}{\Sigma}+O(1),
 &&r=1,
 \label{eq:linear-kernel}\\
 \frac{1}{(2\pi)^2}\int_{\R^2}
 (1-e^{-i\theta\cdot x})\frac{2\chi(\theta)}
 {\theta^T\Sigma\theta}\,d\theta
 &=\frac{1}{\pi\sqrt{\det\Sigma}}
   \log R_\Sigma(x)+O(1),
 &&r=2.
 \label{eq:log-kernel}
\end{align}
If a remainder \(R\) is supported in the same neighborhood, is smooth away
from \(0\), and satisfies
\[
  |\partial^\alpha R(\theta)|
  \leq C_\alpha|\theta|^{-1-|\alpha|}
\]
near \(0\), and obeys \(R(-\theta)=\overline{R(\theta)}\), then its Fourier
coefficient is \(O(R_\Sigma(x)^{1-r})\) for \(r\geq3\).  In dimension two
its contribution to \eqref{eq:log-kernel} is \(O(1)\).  The same conclusion
holds in dimension one provided
\[
  R(\theta)=\frac{ic}{\theta}+O(1)
  \quad(\theta\to0)
\]
for some \(c\in\R\); the integral is then interpreted in the principal-value
sense.
\end{lemma}

\begin{proof}
Let \(A=\Sigma^{1/2}\), chosen positive and symmetric, and set
\[
  u=A\theta,\qquad y=A^{-1}x,\qquad
  \chi_\Sigma(u):=\chi(A^{-1}u).
\]
Then \(R_\Sigma(x)=|y|\), \(d\theta=(\det\Sigma)^{-1/2}du\), and
\(\theta\cdot x=u\cdot y\).  Thus every principal integral is reduced to
the isotropic multiplier \(2|u|^{-2}\), followed by multiplication by
\((\det\Sigma)^{-1/2}\).

\smallskip
\noindent\emph{The transient dimensions \(r\geq3\).}
For \(y\neq0\), the Fourier transform of the full Riesz multiplier can be
computed from
\[
  \frac{2}{|u|^2}=2\int_0^\infty e^{-t|u|^2}\,dt
\]
and the Gaussian Fourier transform.  Namely,
\begin{align*}
 \frac{1}{(2\pi)^r}\int_{\R^r}
 e^{-iu\cdot y}\frac{2}{|u|^2}\,du
 &=2\int_0^\infty
   (4\pi t)^{-r/2}
   e^{-|y|^2/(4t)}\,dt\\
 &=\frac{\Gamma(\frac r2-1)}{2\pi^{r/2}}
   |y|^{2-r}.
\end{align*}
The first identity may equivalently be read in the sense of tempered
distributions.  The last equality follows from
\(s=|y|^2/(4t)\).

It remains to estimate the effect of the cutoff.  Put
\[
  h(u):=\frac{2(\chi_\Sigma(u)-1)}{|u|^2}.
\]
Because \(\chi_\Sigma=1\) near \(0\), \(h\) is smooth there.  At infinity,
\[
  |\partial^\alpha h(u)|
  \leq C_\alpha(1+|u|)^{-2-|\alpha|}.
\]
In particular, every derivative of order \(r-1\) is integrable.  If
\(\omega=y/|y|\), distributional integration by parts \(r-1\) times gives
\begin{align*}
 \left|\frac{1}{(2\pi)^r}
 \int_{\R^r}e^{-iu\cdot y}h(u)\,du\right|
 &=
 \frac{|y|^{1-r}}{(2\pi)^r}
 \left|\int_{\R^r}e^{-iu\cdot y}
   (\omega\cdot\nabla)^{r-1}h(u)\,du\right|\\
 &\leq C|y|^{1-r}.
\end{align*}
Combining this bound with the full Riesz transform, and then restoring the
factor \((\det\Sigma)^{-1/2}\), proves
\eqref{eq:riesz-asymptotic}.

\smallskip
\noindent\emph{The one-dimensional kernel.}
Here \(u=\sqrt\Sigma\,\theta\) and \(y=x/\sqrt\Sigma\).  First,
\begin{align*}
 \frac{1}{2\pi}\operatorname{pv}\!\int_{\R}
 (1-e^{-iuy})\frac{2}{u^2}\,du
 &=\frac{1}{\pi}\int_{\R}\frac{1-\cos(uy)}{u^2}\,du\\
 &=|y|.
\end{align*}
The imaginary part vanishes in the principal-value sense, and the last
identity follows by scaling and
\(\int_{\R}(1-\cos v)v^{-2}\,dv=\pi\).  Moreover,
\[
  \frac{2(\chi_\Sigma(u)-1)}{u^2}\in L^1(\R),
\]
so multiplying it by \(1-e^{-iuy}\) changes the preceding expression by
at most a constant independent of \(y\).  The Jacobian contributes
\(\Sigma^{-1/2}\), and hence
\(\Sigma^{-1/2}|y|=|x|/\Sigma\).  This proves
\eqref{eq:linear-kernel}.

\smallskip
\noindent\emph{The two-dimensional kernel.}
Choose a radial cutoff \(\chi_0(u)=\varphi(|u|)\) which equals \(1\) near
\(0\).  Since
\[
  \frac{\chi_\Sigma-\chi_0}{|u|^2}\in L^1(\R^2),
\]
replacing \(\chi_\Sigma\) by \(\chi_0\) changes the potential integral by
\(O(1)\), uniformly in \(y\).  Polar coordinates, followed by angular
integration, give
\begin{align*}
 &\frac{1}{(2\pi)^2}\int_{\R^2}
 (1-e^{-iu\cdot y})\frac{2\chi_0(u)}{|u|^2}\,du\\
 &\qquad
 =\frac1\pi\int_0^\infty
   \varphi(\rho)\frac{1-J_0(\rho|y|)}{\rho}\,d\rho,
\end{align*}
where \(J_0\) is the Bessel function of order zero.  If
\(\varphi=1\) on \([0,a]\) and is supported in \([0,b]\), then the change
of variables \(s=\rho|y|\) shows that the last integral equals
\[
  \frac1\pi\int_0^{a|y|}
   \frac{1-J_0(s)}{s}\,ds+O(1).
\]
The integral over \(0<s<1\) is finite because
\(1-J_0(s)=O(s^2)\).  On \(s\geq1\), the estimate
\(J_0(s)=O(s^{-1/2})\) implies that
\(\int_1^\infty J_0(s)s^{-1}\,ds\) converges.  Therefore
\[
  \int_0^{a|y|}\frac{1-J_0(s)}s\,ds
  =\log|y|+O(1).
\]
Restoring the determinant factor and using \(|y|=R_\Sigma(x)\) proves
\eqref{eq:log-kernel}.

\smallskip
\noindent\emph{The symbol remainder.}
It suffices to treat the part of \(R\) near zero; a smooth part supported
away from zero has Fourier coefficients \(O_N(|x|^{-N})\) by repeated
integration by parts.  Choose a smooth annular function \(\psi\) giving a
dyadic partition of a punctured neighborhood of zero, and write
\[
  R=\sum_{j\geq j_0}R_j+R_{\mathrm{sm}},
  \qquad
  R_j(\theta):=\psi(2^j\theta)R(\theta).
\]
The symbol bounds imply that
\[
  b_j(v):=2^{-j}\psi(v)R(2^{-j}v)
\]
is supported in one fixed annulus and is bounded there in every
\(C^N\)-norm, uniformly in \(j\).  After the change
\(\theta=2^{-j}v\), integration by parts gives
\begin{equation}\label{eq:dyadic-remainder-bound}
 \left|\frac{1}{(2\pi)^r}\int_{\R^r}
  e^{-i\theta\cdot x}R_j(\theta)\,d\theta\right|
 \leq C_N2^{-j(r-1)}
       (1+2^{-j}|x|)^{-N}.
\end{equation}
Choose \(J\) with \(2^J\leq1+|x|<2^{J+1}\).  For \(j\geq J\), summing
\eqref{eq:dyadic-remainder-bound} without the oscillatory gain gives
\[
  \sum_{j\geq J}2^{-j(r-1)}
  =O(|x|^{1-r}).
\]
For \(j<J\), choose \(N>r-1\); then
\[
 \sum_{j< J}2^{-j(r-1)}
 (2^{-j}|x|)^{-N}
 \leq C|x|^{-N}2^{J(N-r+1)}
 =O(|x|^{1-r}).
\]
This proves the Fourier-coefficient assertion for \(r\geq3\).  Since
\(|x|\) and \(R_\Sigma(x)\) are equivalent norms, the stated anisotropic
bound follows.

When \(r=2\), the estimate \(R(\theta)=O(|\theta|^{-1})\) makes \(R\)
integrable near zero.  Smoothness away from zero and compact support then
give \(R\in L^1(\R^2)\).  Thus
\[
 \left|\int_{\R^2}(1-e^{-i\theta\cdot x})R(\theta)\,d\theta\right|
 \leq2\norm{R}_{L^1}=O(1).
\]
When \(r=1\), choose an even cutoff \(\chi_1\) equal to \(1\) near zero and
write
\[
  R(\theta)=\frac{ic\,\chi_1(\theta)}{\theta}+R_0(\theta),
  \qquad R_0\in L^1(\R).
\]
The \(R_0\)-contribution is bounded.  For the singular part,
\[
 \operatorname{pv}\!\int_{\R}
 (1-e^{-i\theta x})\frac{\chi_1(\theta)}{\theta}\,d\theta
 =2i\int_0^\infty\frac{\sin(\theta x)}{\theta}
   \chi_1(\theta)\,d\theta,
\]
which is uniformly bounded in \(x\) by the Dirichlet test.  This completes
the proof.
\end{proof}

\begin{theorem}
\label{thm:green-potential-asymptotics}
Assume Assumption~\ref{ass:fourier} and \(m=0\).  As
\(|x|\to\infty\) through \(x\in\Z^r\), the following statements hold.
\begin{enumerate}[label=\textup{(\roman*)},leftmargin=2.4em]
\item If \(r\geq3\), then \(G_{\rho_0}(x)<\infty\) and
\begin{equation}\label{eq:green-asymptotic}
  G_{\rho_0}(x)
  =
  \frac{\Gamma(\frac r2-1)}
  {2\pi^{r/2}\sqrt{\det\Sigma}}\,
  R_\Sigma(x)^{2-r}
  +O\!\left(R_\Sigma(x)^{1-r}\right).
\end{equation}

\item If \(r=1\), the series \eqref{eq:potential-kernel} converges and
\begin{equation}\label{eq:potential-r1}
  \mathfrak a_{\rho_0}(x)
  =\frac{|x|}{\Sigma}+O(1).
\end{equation}

\item If \(r=2\), the series \eqref{eq:potential-kernel} converges and
\begin{equation}\label{eq:potential-r2}
  \mathfrak a_{\rho_0}(x)
  =\frac{1}{\pi\sqrt{\det\Sigma}}
   \log R_\Sigma(x)+O(1).
\end{equation}
\end{enumerate}
The error bounds are uniform over all density matrices \(\rho_0\).
\end{theorem}

\begin{proof}
\smallskip
\noindent\emph{Step 1: Abel resolvents.}
For \(0<t<1\), set
\[
  G_{t,\rho_0}(x):=\sum_{n=0}^\infty t^np_n(x).
\]
Since \(p_n(x)\leq1\), this series converges absolutely.  Moreover,
Lemma~\ref{lem:fourier-contraction} gives
\(r(t\Phi_\theta)\leq t<1\), so the operator Neumann series converges in finite
dimension:
\[
  (I-t\Phi_\theta)^{-1}
  =\sum_{n=0}^\infty t^n\Phi_\theta^n.
\]
Using \eqref{eq:fourier-inversion} and absolute convergence, we may
interchange the sum and the torus integral.  We obtain
\begin{equation}\label{eq:abel-resolvent}
  \sum_{n=0}^\infty t^np_n(x)
  =\frac{1}{(2\pi)^r}\int_{\mathbb T^r}
   e^{-i\theta\cdot x}
   F_{t,\rho_0}(\theta)\,d\theta,
  \qquad
  F_{t,\rho_0}(\theta)
  :=\Tr\!\left((I-t\Phi_\theta)^{-1}(\rho_0)\right).
\end{equation}

\smallskip
\noindent\emph{Step 2: the resolvent singularity at \(0\).}
If \(\theta\notin\Theta\), then \(r(\Phi_\theta)<1\), so
\(I-\Phi_\theta\) is invertible.  If
\(\xi\in\Theta\setminus\{0\}\), the only unit-modulus eigenvalue is
\(\zeta_\xi\neq1\); hence \(1\notin\spec(\Phi_\xi)\).  It follows, by
compactness, that \(I-\Phi_\theta\) is uniformly invertible outside every
fixed neighborhood of \(0\).

Near \(0\), use the spectral decomposition from Step~2 of
Theorem~\ref{thm:uniform-llt}.  On the punctured neighborhood,
\begin{align*}
 (I-\Phi_\theta)^{-1}
 &=
 \frac{1}{1-\lambda_0(\theta)}\Pi_{0,\theta}
 +(I-N_{0,\theta})^{-1}(I-\Pi_{0,\theta}).
\end{align*}
Consequently, if
\[
  a_0(\theta;\rho_0)
  :=\Tr\bigl(\Pi_{0,\theta}(\rho_0)\bigr),
\]
then
\begin{equation}\label{eq:resolvent-singularity}
  F_{1,\rho_0}(\theta)
  =\frac{a_0(\theta;\rho_0)}{1-\lambda_0(\theta)}
   +H_{\rho_0}(\theta)
  =\frac{2}{\theta^T\Sigma\theta}+R_{\rho_0}(\theta),
\end{equation}
where
\[
 H_{\rho_0}(\theta)
 :=\Tr\!\left(
 (I-N_{0,\theta})^{-1}
 (I-\Pi_{0,\theta})(\rho_0)\right)
\]
is analytic.  To verify the last equality and its remainder estimate, put
\[
  Q(\theta):=\frac12\theta^T\Sigma\theta.
\]
Because \(m=0\), \eqref{eq:lambda-expansion} and analyticity give
\[
  1-\lambda_0(\theta)=Q(\theta)+O(|\theta|^3),
  \qquad
  a_0(\theta;\rho_0)=1+O(|\theta|).
\]
Since \(Q(\theta)\geq c|\theta|^2\), we have, on the punctured
neighborhood,
\begin{align*}
 &\frac{a_0(\theta;\rho_0)}{1-\lambda_0(\theta)}
 -\frac1{Q(\theta)}\\
 &\quad
 =\frac{a_0(\theta;\rho_0)-1}{1-\lambda_0(\theta)}
 +\frac{Q(\theta)-(1-\lambda_0(\theta))}
 {Q(\theta)(1-\lambda_0(\theta))}
 =O(|\theta|^{-1}).
\end{align*}
Differentiating this quotient shows, for every multi-index \(\alpha\),
\begin{equation}\label{eq:resolvent-remainder-symbol}
  |\partial^\alpha R_{\rho_0}(\theta)|
  \leq C_\alpha|\theta|^{-1-|\alpha|}.
\end{equation}

The relation
\(\Phi_{-\theta}(X^*)=\Phi_\theta(X)^*\) gives
\(F_{1,\rho_0}(-\theta)=\overline{F_{1,\rho_0}(\theta)}\), which is the
required symmetry.  In dimension one, ordinary Taylor division gives
\[
  R_{\rho_0}(\theta)=\frac{c_{-1}(\rho_0)}{\theta}+O(1).
\]
The symmetry forces
\(-c_{-1}(\rho_0)=\overline{c_{-1}(\rho_0)}\), so
\(c_{-1}(\rho_0)=ic_{\rho_0}\) for some \(c_{\rho_0}\in\R\).
All constants in this step are uniform in \(\rho_0\): the spectral
projections and their derivatives are uniformly bounded, and the set of
density matrices is compact.

\smallskip
\noindent\emph{Step 3: the Green function for \(r\geq3\).}
The centered bound \eqref{eq:centered-upper-return}, applied to
\(K=\{x\}\), implies
\[
  G_{\rho_0}(x)=\sum_{n=0}^\infty p_n(x)<\infty
  \qquad(r\geq3).
\]
Since the summands are nonnegative,
\[
  G_{t,\rho_0}(x)\uparrow G_{\rho_0}(x)
  \quad\text{as }t\uparrow1.
\]
We next justify passage to the limit in the Fourier integral.  From
\eqref{eq:lambda-gaussian-majorant}, for \(t\) sufficiently close to \(1\),
\begin{align*}
  |1-t\lambda_0(\theta)|
  &\geq1-t|\lambda_0(\theta)|\\
  &\geq(1-t)+c|\theta|^2.
\end{align*}
The complementary resolvent is uniformly bounded, and hence
\begin{equation}\label{eq:abel-resolvent-bound}
  |F_{t,\rho_0}(\theta)|
  \leq C\left(1+\frac1{(1-t)+|\theta|^2}\right)
\end{equation}
near \(0\), uniformly in \(\rho_0\).  When \(r\geq3\), the function
\(|\theta|^{-2}\) is locally integrable.  Dominated convergence therefore
gives
\begin{equation}\label{eq:green-fourier-resolvent}
  G_{\rho_0}(x)
  =\frac{1}{(2\pi)^r}\int_{\mathbb T^r}
   e^{-i\theta\cdot x}F_{1,\rho_0}(\theta)\,d\theta.
\end{equation}

Choose the cutoff \(\chi\) of
Lemma~\ref{lem:fourier-singularity}, viewed in a coordinate neighborhood of
\(0\) in the torus.  Equation~\eqref{eq:resolvent-singularity} gives
\[
 F_{1,\rho_0}(\theta)
 =\chi(\theta)\frac{2}{\theta^T\Sigma\theta}
  +\chi(\theta)R_{\rho_0}(\theta)
  +S_{\rho_0}(\theta),
\]
where \(S_{\rho_0}\) is smooth on \(\mathbb T^r\).  Its Fourier
coefficients are \(O_N(|x|^{-N})\), uniformly in \(\rho_0\).
Lemma~\ref{lem:fourier-singularity}, together with
\eqref{eq:resolvent-remainder-symbol}, now yields
\eqref{eq:green-asymptotic}.

\smallskip
\noindent\emph{Step 4: convergence of the ordinary potential series.}
Let
\[
  d_n(x):=p_n(0)-p_n(x).
\]
Fourier inversion gives
\begin{equation}\label{eq:potential-coefficient-fourier}
 d_n(x)
 =\frac{1}{(2\pi)^r}\int_{\mathbb T^r}
  (1-e^{-i\theta\cdot x})
  \Tr\bigl(\Phi_\theta^n(\rho_0)\bigr)\,d\theta.
\end{equation}
We split the integral into the peripheral neighborhoods \(V_\xi\) and their
complement.  The complement and the spectral remainders contribute
\(O(\kappa^n)\), by
\eqref{eq:local-remainder-bound}--\eqref{eq:away-origin-bound}.

For \(\xi\in\Theta\), denote by \(d_{n,\xi}(x)\) the contribution of the
rank-one spectral term on \(V_\xi\).  With
\begin{align*}
 \psi_\xi(\eta)
 &:=\log\!\left(\zeta_\xi^{-1}
        \lambda_\xi(\xi+\eta)\right),\\
 h_{\xi,x}(\eta)
 &:=(1-e^{-i(\xi+\eta)\cdot x})
   a_\xi(\eta;\rho_0),
\end{align*}
we have
\begin{equation}\label{eq:peripheral-potential-contribution}
 d_{n,\xi}(x)
 =\frac{\zeta_\xi^n}{(2\pi)^r}
   \int_{|\eta|<\delta}
   h_{\xi,x}(\eta)e^{n\psi_\xi(\eta)}\,d\eta.
\end{equation}

We first treat \(\xi=0\).  Since \(a_0(0;\rho_0)=1\), Taylor's theorem
gives, uniformly over density matrices,
\begin{equation}\label{eq:zero-mode-amplitude-expansion}
 h_{0,x}(\eta)=i\eta\cdot x+O_x(|\eta|^2).
\end{equation}
Put \(Q(u):=\frac12u^T\Sigma u\), set \(u=\sqrt n\,\eta\), and fix
\(\varepsilon=1/12\).  On the inner region
\(|u|\leq n^\varepsilon\), analyticity and
\eqref{eq:lambda-expansion} imply
\begin{align}
 h_{0,x}(u/\sqrt n)
 &=n^{-1/2}i\,u\cdot x
   +O_x(n^{-1}|u|^2),
 \label{eq:scaled-zero-amplitude}\\
 n\psi_0(u/\sqrt n)
 &=-Q(u)+O(n^{-1/2}|u|^3).
 \label{eq:scaled-zero-phase}
\end{align}
Because \(n^{-1/2}|u|^3\leq n^{-1/4}\) on this region, the elementary
estimate \(|e^z-1|\leq |z|e^{|z|}\) yields, after reducing \(\delta\) if
necessary,
\begin{equation}\label{eq:scaled-zero-exponential-error}
 \left|e^{n\psi_0(u/\sqrt n)}-e^{-Q(u)}\right|
 \leq Cn^{-1/2}|u|^3e^{-c_1|u|^2}
\end{equation}
for some \(c_1>0\).  The integral of
\(n^{-1/2}i\,u\cdot x\,e^{-Q(u)}\) over the symmetric ball
\(|u|\leq n^\varepsilon\) vanishes.  Equations
\eqref{eq:scaled-zero-amplitude}--
\eqref{eq:scaled-zero-exponential-error} therefore show that the remaining
inner integral is bounded by
\[
 C_xn^{-1}\int_{\R^r}(|u|^2+|u|^4)e^{-c_1|u|^2}\,du
 =O_x(n^{-1}).
\]
On the outer region
\(n^\varepsilon<|u|<\delta\sqrt n\), the Gaussian majorant
\eqref{eq:lambda-gaussian-majorant} and boundedness of the amplitude give
\[
 \int_{n^\varepsilon<|u|<\delta\sqrt n}
 |h_{0,x}(u/\sqrt n)|
 |e^{n\psi_0(u/\sqrt n)}|\,du
 \leq Ce^{-c_2n^{2\varepsilon}}.
\]
Restoring the Jacobian \(n^{-r/2}\) in
\eqref{eq:peripheral-potential-contribution} proves
\begin{equation}\label{eq:zero-mode-potential-bound}
  d_{n,0}(x)=O_x\!\left(n^{-(r+2)/2}\right).
\end{equation}

We next treat \(\xi\neq0\).  Decrease \(\delta\), if necessary, so that
\[
  |1-\lambda_\xi(\xi+\eta)|\geq c_\xi>0
  \qquad(|\eta|<\delta);
\]
this is possible because \(\lambda_\xi(\xi)=\zeta_\xi\neq1\).
The finite geometric-sum identity gives
\begin{align*}
 \sum_{n=1}^N d_{n,\xi}(x)
 &=\frac1{(2\pi)^r}\int_{|\eta|<\delta}
 h_{\xi,x}(\eta)
 \frac{\lambda_\xi(\xi+\eta)
       (1-\lambda_\xi(\xi+\eta)^N)}
      {1-\lambda_\xi(\xi+\eta)}\,d\eta.
\end{align*}
The term containing \(\lambda_\xi^N\) is bounded in modulus by
\[
 C_x\int_{|\eta|<\delta}e^{-cN|\eta|^2}\,d\eta
 =O_x(N^{-r/2}),
\]
uniformly in \(\rho_0\).  Thus each nonzero peripheral contribution has
convergent ordinary partial sums.  Together with the absolutely summable
zero-mode bound \eqref{eq:zero-mode-potential-bound} and the exponentially
summable complement, this proves convergence of
\(\sum_{n=0}^\infty d_n(x)\) for \(r=1,2\), uniformly over \(\rho_0\)
for each fixed \(x\).

\smallskip
\noindent\emph{Step 5: identification and asymptotics of the potential.}
For \(0<t<1\), absolute convergence and
\eqref{eq:abel-resolvent} give
\begin{equation}\label{eq:abel-potential-resolvent}
 \sum_{n=0}^\infty t^nd_n(x)
 =\frac{1}{(2\pi)^r}\int_{\mathbb T^r}
  (1-e^{-i\theta\cdot x})
  F_{t,\rho_0}(\theta)\,d\theta.
\end{equation}
Since the ordinary series converges, Abel's theorem implies that its
left-hand side tends to \(\mathfrak a_{\rho_0}(x)\) as \(t\uparrow1\).

We describe the limit on the right.  Away from \(0\), the resolvents converge
smoothly and uniformly.  In dimension two,
\eqref{eq:abel-resolvent-bound} and
\(|1-e^{-i\theta\cdot x}|\leq C_x|\theta|\) give the locally integrable
majorant \(C_x|\theta|^{-1}\), so dominated convergence applies.  In
dimension one, pair the integrands at \(\theta\) and \(-\theta\).  The
symmetry
\[
  F_{t,\rho_0}(-\theta)
  =\overline{F_{t,\rho_0}(\theta)}
\]
shows that the paired integrand is
\[
  2\operatorname{Re}\!\left[
   (1-e^{-i\theta x})F_{t,\rho_0}(\theta)\right].
\]
The Taylor expansions
\[
 \lambda_0(\theta)
 =1-\frac{\Sigma}{2}\theta^2+O(\theta^3),
 \qquad
 a_0(\theta;\rho_0)=1+O(\theta),
\]
together with the symmetry, give the more precise one-dimensional
expansions
\[
 a_0(\theta;\rho_0)=1+i\alpha_{\rho_0}\theta+O(\theta^2),
 \qquad
 1-t\lambda_0(\theta)
 =\delta_t+\frac{t\Sigma}{2}\theta^2
  +it\beta\theta^3+O(\theta^4),
\]
where \(\delta_t=1-t\), \(\alpha_{\rho_0},\beta\in\R\), and all remainders
are uniform for \(t\) close to \(1\).  If
\[
  D_t(\theta):=\delta_t+\frac{t\Sigma}{2}\theta^2,
\]
Taylor division gives
\[
 F_{t,\rho_0}(\theta)
 =\frac1{D_t(\theta)}
  +\frac{i\alpha_{\rho_0}\theta}{D_t(\theta)}
  +\frac{i\beta_t\theta^3}{D_t(\theta)^2}
  +E_{t,\rho_0}(\theta),
\]
where \(\beta_t=-t\beta\) and \(E_{t,\rho_0}\) is uniformly bounded;
indeed, terms such as
\(\theta^2/D_t\), \(\theta^4/D_t^2\), and
\(\theta^6/D_t^3\) are bounded by constants.  In the paired expression,
the first term contributes
\[
  \frac{2(1-\cos(\theta x))}{D_t(\theta)}
  =O_x\!\left(\frac{\theta^2}{\delta_t+\theta^2}\right),
\]
the second contributes
\[
  O_x\!\left(\frac{\theta|\sin(\theta x)|}
  {\delta_t+\theta^2}\right)=O_x(1),
\]
and the third contributes
\[
  O_x\!\left(\frac{\theta^3|\sin(\theta x)|}
  {(\delta_t+\theta^2)^2}\right)=O_x(1).
\]
The bounded remainder causes no difficulty.  Thus the paired integrand is
bounded by a constant depending on \(x\), uniformly in \(t\) and
\(\theta\).  Dominated convergence on the positive half-interval gives the
principal-value limit
\begin{equation}\label{eq:ordinary-potential-resolvent}
 \mathfrak a_{\rho_0}(x)
 =\frac{1}{(2\pi)^r}\operatorname{pv}
  \int_{\mathbb T^r}(1-e^{-i\theta\cdot x})
  F_{1,\rho_0}(\theta)\,d\theta,
\end{equation}
where the principal value is needed only for \(r=1\).

Finally, insert the cutoff decomposition from Step~3 into
\eqref{eq:ordinary-potential-resolvent}.  The smooth term contributes
\(O(1)\), uniformly in \(x\) and \(\rho_0\).  The localized remainder
\(\chi R_{\rho_0}\) also contributes \(O(1)\) by the last part of
Lemma~\ref{lem:fourier-singularity}.  The principal quadratic term is
exactly the integral in \eqref{eq:linear-kernel} for \(r=1\) and in
\eqref{eq:log-kernel} for \(r=2\).  This proves
\eqref{eq:potential-r1}--\eqref{eq:potential-r2}.

All error constants are uniform over \(\rho_0\): the local symbol bounds,
the smooth-resolvent bounds, and the ordinary-series tail bounds depend
on \(\rho_0\) only through bounded linear functionals of a density matrix.
\end{proof}

\begin{example}
\label{ex:periodic-simple}
Let \(\cH=\C\), \(S=\{-1,1\}\), and
\(L_{-1}=L_1=2^{-1/2}\).  Then
\[
  L_{-1}^*L_{-1}+L_1^*L_1=1,
\]
so these operators define an OQRW.  Since the internal space is
one-dimensional, its only density matrix is \(\rho_0=1\), the local channel
is the identity on \(\C\), and the measured position process is precisely
the classical simple symmetric random walk.  In particular,
\begin{equation}\label{eq:periodic-exact-law}
 p_n(x)=
 \begin{cases}
  2^{-n}\displaystyle\binom{n}{(n+x)/2},
    & |x|\leq n\text{ and }n+x\text{ is even},\\[1ex]
  0,&\text{otherwise}.
 \end{cases}
\end{equation}

The Fourier deformation acts by scalar multiplication:
\[
  \Phi_\theta(z)
  =\frac12(e^{-i\theta}+e^{i\theta})z
  =\cos\theta\,z.
\]
Consequently,
\[
  \Theta=\{0,\pi\},\qquad
  \zeta_0=1,\quad \zeta_\pi=-1.
\]
Here \(\pi\) denotes the nonzero point of the torus
\(\R/(2\pi\Z)\).  At the two peripheral points,
\begin{align*}
 \log\lambda_0(\eta)
 &=\log(\cos\eta)
   =-\frac{\eta^2}{2}+O(\eta^4),\\
 \log\!\left(\zeta_\pi^{-1}\lambda_\pi(\pi+\eta)\right)
 &=\log(\cos\eta)
   =-\frac{\eta^2}{2}+O(\eta^4).
\end{align*}
Thus both branches have the common parameters
\[
  m=0,\qquad \Sigma=1.
\]
The spectral projections are the identity on \(\C\), so
\(a_0=a_\pi=1\).  Since \(e^{-i\pi x}=(-1)^x\), the periodic mask is
\[
  D_{n,x}
  =1+(-1)^ne^{-i\pi x}
  =1+(-1)^{n+x}
  =\begin{cases}
     2,&n+x\text{ is even},\\
     0,&n+x\text{ is odd}.
   \end{cases}
\]
Theorem~\ref{thm:uniform-llt} therefore yields the uniform periodic local
limit
\[
 p_n(x)
 =\frac{D_{n,x}}{\sqrt{2\pi n}}
   \exp\!\left(-\frac{x^2}{2n}\right)
  +o(n^{-1/2}),
\]
where the error is uniform in \(x\in\Z\).  Thus the Gaussian term is
multiplied by \(2\) on the admissible parity class and vanishes on the other
class, in agreement with the exact law \eqref{eq:periodic-exact-law}.

At the origin, only even times contribute and
\[
  p_{2k}(0)=4^{-k}\binom{2k}{k}
  \sim\frac1{\sqrt{\pi k}},
  \qquad p_{2k+1}(0)=0.
\]
Stirling's formula and summation of \(k^{-1/2}\) give
\[
  \sum_{n=0}^Np_n(0)
  \sim\sqrt{\frac{2N}{\pi}},
\]
so the origin is recurrent.

The potential kernel can also be evaluated exactly.  Theorem
\ref{thm:green-potential-asymptotics} guarantees convergence of its ordinary
series.  Since \(\Phi_\theta\) is multiplication by \(\cos\theta\), its
scalar resolvent is \(F_{1}(\theta)=(1-\cos\theta)^{-1}\).  Substitution
in \eqref{eq:ordinary-potential-resolvent} gives
\begin{align*}
 \mathfrak a(x)
 &=\frac1{2\pi}\operatorname{pv}\!\int_{-\pi}^{\pi}
   \frac{1-e^{-i\theta x}}{1-\cos\theta}\,d\theta\\
 &=\frac1{2\pi}\int_{-\pi}^{\pi}
   \frac{1-\cos(x\theta)}{1-\cos\theta}\,d\theta.
\end{align*}
The imaginary part vanishes by oddness.  For \(x\neq0\),
\[
 \frac{1-\cos(x\theta)}{1-\cos\theta}
 =\left(\frac{\sin(x\theta/2)}{\sin(\theta/2)}\right)^2,
\]
For \(k=|x|\), the squared ratio equals
\(\left|\sum_{j=0}^{k-1}e^{ij\theta}\right|^2\).  Expanding the square
and using orthogonality of the exponentials gives its integral over one
period as \(2\pi k\).  Hence
\[
  \mathfrak a(x)=|x|\qquad(x\in\Z).
\]
This example satisfies the periodic hypotheses with two peripheral modes,
but it is excluded by strong aperiodicity because
\(r(\Phi_\pi)=1\).
\end{example}

The admissible space--time parity classes in this example are displayed in
Figure~\ref{fig:oqrw-periodic-mask}.

\FloatBarrier
\begin{figure}[t]
\centering
\resizebox{\textwidth}{!}{%
\begin{tikzpicture}[x=0.64cm,y=0.78cm]
  \draw[->,thick] (-6.8,0) -- (7.0,0) node[right] {$x$};
  \draw[->,thick] (0,-0.35) -- (0,7.0) node[above] {$n$};

  \foreach \x in {-6,-5,...,6}{
    \draw[oqrwgray!25] (\x,0) -- (\x,6.25);
    \node[below,font=\scriptsize] at (\x,0) {$\x$};
  }
  \foreach \n in {0,1,...,6}{
    \draw[oqrwgray!25] (-6.25,\n) -- (6.25,\n);
  }
  \foreach \n in {1,...,6}{
    \node[left,font=\scriptsize] at (0,\n) {$\n$};
  }

  \foreach \n in {0,1,...,6}{
    \foreach \x in {-6,-5,...,6}{
      \pgfmathtruncatemacro{\parity}{mod(\n+\x+20,2)}
      \ifnum\parity=0
        \fill[oqrwblue] (\x,\n) circle (2.2pt);
      \else
        \draw[oqrwgray!60,fill=white] (\x,\n) circle (1.8pt);
      \fi
    }
  }

  \node[oqrw blue,text width=5.1cm,anchor=west] at (7.6,4.9) {
    \textbf{Periodic mask}\\[2pt]
    $D_{n,x}=1+(-1)^{n+x}$
  };
  \node[oqrw green,text width=5.1cm,anchor=west] at (7.6,3.0) {
    filled: $n+x$ even\\
    $D_{n,x}=2$
  };
  \node[oqrw gray,text width=5.1cm,anchor=west] at (7.6,1.2) {
    open: $n+x$ odd\\
    $D_{n,x}=0$
  };

  \draw[oqrw dashed arrow] (6.2,5.8) -- (7.55,5.1);
\end{tikzpicture}%
}
\caption{Parity mask for the one-dimensional simple symmetric walk.  The
periodic local theorem retains both peripheral Fourier points and produces
the factor two on the admissible parity class, rather than imposing strong
aperiodicity.}
\label{fig:oqrw-periodic-mask}
\end{figure}
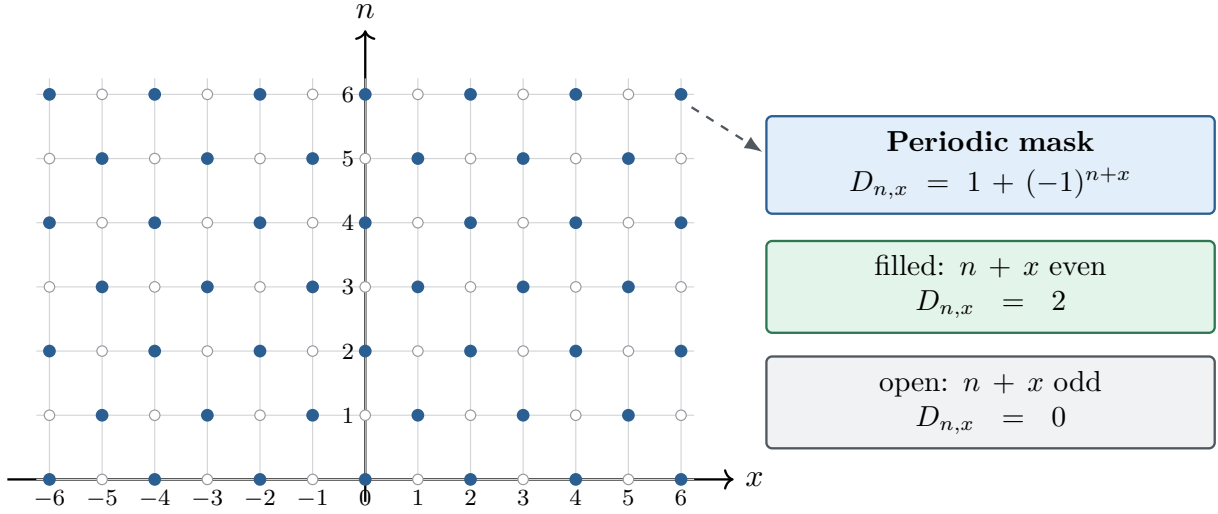
%\FloatBarrier

\begin{example}
\label{ex:explicit-primitive}
Let \(r=1\), \(\cH=\C^2\), and write
\(E_{ab}=|e_a\rangle\langle e_b|\) for the standard matrix units.  On the
jump set \(S=\{-1,0,1,2\}\), define
\[
  L_0=2^{-1/2}E_{00},\qquad
  L_2=2^{-1/2}E_{01},\qquad
  L_{-1}=2^{-1/2}E_{10},\qquad
  L_1=2^{-1/2}E_{11}.
\]
The normalization can be checked componentwise:
\begin{align*}
 L_0^*L_0+L_{-1}^*L_{-1}&=E_{00},\\
 L_2^*L_2+L_1^*L_1&=E_{11}.
\end{align*}
Thus \(\sum_sL_s^*L_s=I\).

For \(X=(x_{ab})_{a,b=0}^1\), direct multiplication gives
\begin{align*}
 L_0XL_0^*&=\frac{x_{00}}2E_{00},&
 L_2XL_2^*&=\frac{x_{11}}2E_{00},\\
 L_{-1}XL_{-1}^*&=\frac{x_{00}}2E_{11},&
 L_1XL_1^*&=\frac{x_{11}}2E_{11}.
\end{align*}
Consequently,
\[
  \cL(X)=\Tr(X)\frac{I}{2}.
\]
In particular, \(\rho_\infty=I/2\), and every nonzero positive operator is
sent in one step to a positive-definite operator.  Hence \(\cL\) is
primitive, with primitivity exponent one.

The Fourier deformation annihilates the off-diagonal matrix units.  On a
diagonal matrix \(X=\diag(x_{00},x_{11})\), it acts as
\[
 \Phi_\theta(X)
 =\frac12\bigl(x_{00}+e^{2i\theta}x_{11}\bigr)E_{00}
  +\frac12\bigl(e^{-i\theta}x_{00}+e^{i\theta}x_{11}\bigr)E_{11}.
\]
Thus its restriction to the diagonal subspace is represented by
\[
  Q(\theta)=\frac12
  \begin{pmatrix}
    1&e^{2i\theta}\\
    e^{-i\theta}&e^{i\theta}
  \end{pmatrix},
\]
whose determinant is zero and whose trace is
\((1+e^{i\theta})/2\).  Therefore
\[
  \spec(\Phi_\theta)
  =\left\{\frac{1+e^{i\theta}}2,0\right\}.
\]
The nonzero eigenvalue can be written as
\begin{equation}\label{eq:explicit-eigenvalue-factorization}
 \lambda(\theta)
 =\frac{1+e^{i\theta}}2
 =e^{i\theta/2}\cos(\theta/2).
\end{equation}
It follows that
\[
  r(\Phi_\theta)=\left|\cos\frac\theta2\right|<1
  \quad\text{for }\theta\neq0\text{ in }[-\pi,\pi],
\]
so the peripheral set is \(\Theta=\{0\}\).  This verifies strong
aperiodicity.  Taking the analytic logarithm of
\eqref{eq:explicit-eigenvalue-factorization} near zero gives
\[
  \log\frac{1+e^{i\theta}}2
  =\frac{i\theta}{2}+\log\cos\frac\theta2
  =\frac{i\theta}{2}-\frac{\theta^2}{8}+O(\theta^4).
\]
Assumption~\ref{ass:fourier} holds with
\[
  m=\frac12,
  \qquad
  \Sigma=\frac14.
\]

We can also see these parameters directly from the trajectory.  Write the
initial density matrix as
\[
 \rho_0=\begin{pmatrix}1-q&z\\ \overline z&q\end{pmatrix},
 \qquad 0\leq q\leq1.
\]
The off-diagonal entry \(z\) does not affect any path probability.  Introduce
an initial label \(A_0\in\{0,1\}\) with
\(\PP(A_0=1)=q\).  After one Kraus operation the normalized internal state
is one of the basis states, and the possible transitions are
\begin{center}
\begin{tabular}{@{}cccc@{}}
\toprule
current label & next label & displacement & probability\\
\midrule
\(0\)&\(0\)&\(0\)&\(1/2\)\\
\(0\)&\(1\)&\(-1\)&\(1/2\)\\
\(1\)&\(0\)&\(2\)&\(1/2\)\\
\(1\)&\(1\)&\(1\)&\(1/2\)\\
\bottomrule
\end{tabular}
\end{center}
Thus \(A_1,A_2,\ldots\) are independent Bernoulli variables with parameter
\(1/2\), independent of \(A_0\), and the displacement at step \(k\) is
\[
  X_k-X_{k-1}=2A_{k-1}-A_k.
\]
Telescoping the sum of the displacements gives the exact representation
\begin{equation}\label{eq:explicit-position-representation}
  X_n=2A_0+\sum_{k=1}^{n-1}A_k-A_n
  \qquad(n\geq1).
\end{equation}
In particular,
\begin{align}
 \E_{\rho_0}X_n
 &=\frac n2+2q-1,
 \label{eq:explicit-first-moment}\\
 \operatorname{Var}_{\rho_0}(X_n)
 &=\frac n4+4q(1-q).
 \label{eq:explicit-variance}
\end{align}
The dependence on the initial state is confined to bounded terms.
Chebyshev's inequality and
\eqref{eq:explicit-first-moment}--\eqref{eq:explicit-variance} give
\[
 \frac{X_n}{n}\longrightarrow\frac12
 \quad\text{in probability}.
\]
For the central limit theorem, use
\eqref{eq:explicit-position-representation}: the centered sum
\(\sum_{k=1}^{n-1}(A_k-1/2)\) satisfies the classical i.i.d. central
limit theorem, and the remaining term \(2A_0-A_n-1/2\) is uniformly
bounded.  Slutsky's theorem therefore gives
\[
 \frac{X_n-n/2}{\sqrt n}\Longrightarrow
 \mathcal N\!\left(0,\frac14\right).
\]
The moment formulas also yield
\[
 \frac1{n^2}\E_{\rho_0}X_n^2
 =\frac1{n^2}\left[
   \left(\frac n2+2q-1\right)^2
   +\frac n4+4q(1-q)
  \right]
 \longrightarrow\frac14.
\]
Hence \(v_{\mathrm{bal}}=1/2\), uniformly in the initial density matrix.

For completeness, \eqref{eq:explicit-position-representation} also gives an
exact characteristic function.  For \(n\geq1\),
\begin{equation}\label{eq:explicit-characteristic-function}
 \E_{\rho_0}e^{i\theta X_n}
 =\bigl((1-q)+qe^{2i\theta}\bigr)
  \left(\frac{1+e^{i\theta}}2\right)^{n-1}
  \frac{1+e^{-i\theta}}2.
\end{equation}
This independently displays the dominant eigenvalue
\eqref{eq:explicit-eigenvalue-factorization}.  If
\(B_{n-1}\sim\operatorname{Bin}(n-1,1/2)\), then
\eqref{eq:explicit-position-representation} yields
\begin{equation}\label{eq:explicit-position-law}
 p_n(x)
 =2^{-n}\sum_{a,b\in\{0,1\}}
   q^a(1-q)^{1-a}
   \binom{n-1}{x-2a+b},
\end{equation}
where the Bernoulli weights are interpreted by continuity at \(q=0,1\),
and a binomial coefficient is understood to be zero when its lower index
does not belong to \(\{0,\ldots,n-1\}\).  For \(x\) in a fixed finite set,
all nonzero lower indices on the right remain bounded as \(n\to\infty\).
Thus \(p_n(x)\leq C_xn^{M_x}2^{-n}\), and consequently, for every finite
\(K\subset\Z\),
\[
  \PP_{\rho_0}(X_n\in K)\leq C_Ke^{-c_Kn}.
\]
This elementary estimate agrees with the exponential finite-set return bound
in Theorem~\ref{thm:uniform-llt}.
\end{example}

\subsection{A centered noncommuting family in every dimension}

\begin{example}
\label{ex:centered-noncommuting}
We give a family satisfying every part of
Assumption~\ref{ass:fourier}, including the exclusion of nonzero stationary
Fourier modes.  Begin on \(\Z\), with \(\cH=\C^2\), and put
\[
 P=\begin{pmatrix}1&0\\0&0\end{pmatrix},\qquad
 Q=\frac12\begin{pmatrix}1&1\\1&1\end{pmatrix},\qquad
 H=\left(I-\frac{P+Q}{2}\right)^{1/2}.
\]
The eigenvalues of \(H^2\) are
\((1\pm1/\sqrt2)/2\), so \(H\) is positive definite.  Define
\begin{equation}\label{eq:centered-noncommuting-kraus}
 L_{-1}=L_1=\frac12P,\qquad
 L_{-2}=L_2=\frac12Q,\qquad L_0=H.
\end{equation}
Since \(P^2=P\) and \(Q^2=Q\),
\[
 \sum_{s=-2}^2L_s^*L_s=H^2+\frac12P+\frac12Q=I.
\]
All Kraus operators are positive, and
\([L_1,L_2]=[P,Q]/4\neq0\).  Their jump probabilities depend on the
internal state:
\[
 \PP(\Delta X=\pm1\mid\rho)=\frac14\Tr(P\rho),\qquad
 \PP(\Delta X=\pm2\mid\rho)=\frac14\Tr(Q\rho).
\]
Each displayed probability is for one specified sign.  These increments
are not generally independent.  For example, starting with \(\rho_0=I/2\),
the probability of the first increment being \(1\) is \(1/8\), but its
occurrence leaves the internal state \(P\), making the conditional
probability of the next increment being \(1\) equal to \(1/4\).

\smallskip
\noindent\emph{Primitivity.}
Equip \(B(\C^2)\) with the Hilbert--Schmidt inner product
\(\langle X,Y\rangle_{\mathrm{HS}}=\Tr(X^*Y)\), and let
\(\norm{\cdot}_{\mathrm{HS}}\) be its norm.  For a positive matrix \(A\),
write \(\mathcal A_A(X)=AXA\).  This is a self-adjoint positive
operator on the Hilbert--Schmidt space, since
\[
 \langle X,\mathcal A_A(X)\rangle_{\mathrm{HS}}
 =\norm{A^{1/2}XA^{1/2}}_{\mathrm{HS}}^2\geq0.
\]
The local channel is
\(\cL=\mathcal A_H+\tfrac12\mathcal A_P+\tfrac12\mathcal A_Q\).
It is unital as well as trace preserving, and its faithful invariant state
is \(\rho_\infty=I/2\).  The normalization and self-adjointness of the
Kraus operators give the identity
\begin{equation}\label{eq:centered-dirichlet-identity}
 \langle X,(\mathrm{id}-\cL)X\rangle_{\mathrm{HS}}
 =\frac12\sum_{s=-2}^2
   \norm{L_sX-XL_s}_{\mathrm{HS}}^2.
\end{equation}
The only matrices commuting with both \(P\) and \(Q\) are scalar
matrices.  Thus the fixed space of \(\cL\) is \(\C I\).
Self-adjointness, positivity on Hilbert--Schmidt space, and
\eqref{eq:centered-dirichlet-identity} imply that the remaining eigenvalues
belong to \([0,1)\), and that the eigenvalue \(1\) is algebraically
simple.  Hence \(\cL^n(X)\to\Tr(X)I/2\) in operator norm on the
finite-dimensional matrix space.  Uniform convergence over density
matrices implies that some power sends every density matrix to a positive
definite matrix.  This proves primitivity and verifies \textup{(A1)}.

\smallskip
\noindent\emph{The full Fourier peripheral set.}
For real \(\theta\),
\begin{equation}\label{eq:centered-explicit-fourier}
 \Phi_\theta
 =\mathcal A_H+\frac{\cos\theta}{2}\mathcal A_P
                 +\frac{\cos(2\theta)}{2}\mathcal A_Q
\end{equation}
is self-adjoint on Hilbert--Schmidt space.  Its spectrum lies in
\([-1,1]\) by Lemma~\ref{lem:fourier-contraction}.  If
\(\Phi_\theta(X)=X\), then
\begin{align*}
 0={}&\langle X,(\mathrm{id}-\cL)X\rangle_{\mathrm{HS}}\\
 &+\frac{1-\cos\theta}{2}
        \norm{PXP}_{\mathrm{HS}}^2
  +\frac{1-\cos(2\theta)}{2}
        \norm{QXQ}_{\mathrm{HS}}^2.
\end{align*}
All three terms are nonnegative.  The first forces \(X=cI\);
if \(X\neq0\), the second then forces \(\cos\theta=1\).
Therefore \(1\) is not an eigenvalue at any nonzero point of the torus.
To exclude \(-1\), let \(h_*>0\) be the smallest eigenvalue of \(H\).
The decomposition
\[
 \mathrm{id}+\Phi_\theta
 =(\mathrm{id}-\cL)+2\mathcal A_H
   +\frac{1+\cos\theta}{2}\mathcal A_P
   +\frac{1+\cos(2\theta)}{2}\mathcal A_Q
\]
gives
\[
 \langle X,(\mathrm{id}+\Phi_\theta)X\rangle_{\mathrm{HS}}
 \geq2\norm{H^{1/2}XH^{1/2}}_{\mathrm{HS}}^2
 \geq2h_*^2\norm{X}_{\mathrm{HS}}^2.
\]
Thus \(-1\) is never an eigenvalue.  We have proved
\(\Theta=\{0\}\), with the simple peripheral eigenvalue
\(\zeta_0=1\), establishing \textup{(A2)}.

\smallskip
\noindent\emph{Drift and covariance.}
The analytic principal eigenvalue \(\lambda(\theta)\) is even because
\(\Phi_{-\theta}=\Phi_\theta\).  Moreover,
\(\Phi'_0=0\), so differentiating the normalized eigenvalue equation
and taking traces yields
\[
 \lambda'(0)=0,\qquad
 \lambda''(0)
 =\Tr\bigl(\Phi''_0(I/2)\bigr)
 =\Tr\left(-\frac14P-Q\right)=-\frac54.
\]
Consequently,
\begin{equation}\label{eq:centered-explicit-expansion}
 \log\lambda(\theta)=-\frac58\theta^2+O(\theta^4),
 \qquad m=0,\qquad \Sigma=\frac54.
\end{equation}
This verifies \textup{(A3)}.  Symmetry of each pair of Kraus operators
also makes the conditional mean of each displacement zero.

\smallskip
\noindent\emph{Tensor products and quantitative potentials.}
For any \(r\geq1\), use the internal space \((\C^2)^{\otimes r}\),
jump set \(\{-2,-1,0,1,2\}^r\), and operators
\[
 L_{\boldsymbol s}^{(r)}
 =L_{s_1}\otimes\cdots\otimes L_{s_r}.
\]
Their normalization follows by tensoring the one-dimensional normalization.
The active jumps generate \(\Z^r\), and
\[
 \cL^{(r)}=\cL^{\otimes r},\qquad
 \Phi_{\boldsymbol\theta}^{(r)}
 =\Phi_{\theta_1}\otimes\cdots\otimes\Phi_{\theta_r}.
\]
The tensor-product spectrum proves primitivity with invariant state
\(I/2^r\).  Since each factor has spectral radius at most one, with
equality only at zero, the full peripheral set is again \(\{0\}\).
Near zero the principal eigenvalue is \(\prod_{j=1}^r\lambda(\theta_j)\),
so \eqref{eq:centered-explicit-expansion} gives
\[
 m=0,\qquad \Sigma=\frac54 I_r.
\]
In particular, all hypotheses of Assumption~\ref{ass:fourier} hold in
every dimension.  The Kraus family remains noncommuting: the commutator of
\(L_1\otimes H^{\otimes(r-1)}\) and
\(L_2\otimes H^{\otimes(r-1)}\) is nonzero.

For every internal density matrix at the origin, including entangled
initial states when \(r>1\), Theorem~\ref{thm:uniform-llt} gives
\[
 \sup_{x\in\Z^r}\left|
 n^{r/2}p_n(x)
 -\left(\frac{2}{5\pi}\right)^{r/2}
     \exp\left(-\frac{2|x|^2}{5n}\right)\right|\longrightarrow0.
\]
Each coordinate has conditionally symmetric increments.  Thus \(X_n\)
is a vector-valued martingale and, since \(|X_{n+1}-X_n|^2\leq4r\),
\(\E|X_n|^2\leq4rn\).  Hence \(v_{\mathrm{bal}}=0\).
The cumulative returns satisfy
\[
 \sum_{n=1}^N p_n(0)\sim
 \begin{cases}
  2\sqrt{\dfrac{2}{5\pi}}\,\sqrt N,&r=1,\\[3pt]
  \dfrac{2}{5\pi}\log N,&r=2.
 \end{cases}
\]
Corollary~\ref{cor:tom-origin-recurrence} therefore gives TOM recurrence
of \(P_0\) for \(r=1,2\); for \(r\geq3\) every finite-set occupation
series is finite.  Finally,
Theorem~\ref{thm:green-potential-asymptotics} yields, as \(|x|\to\infty\),
\begin{align*}
 \mathfrak a_{\rho_0}(x)&=\frac45|x|+O(1),&&r=1,\\
 \mathfrak a_{\rho_0}(x)&=\frac{4}{5\pi}\log|x|+O(1),&&r=2,\\
 G_{\rho_0}(x)&=\frac{2\Gamma(\frac r2-1)}{5\pi^{r/2}}
                 |x|^{2-r}+O(|x|^{1-r}),&&r\geq3.
\end{align*}
All limits and error bounds have the initial-state uniformity asserted in
the cited theorems.  Thus the centered local and potential estimates apply
to a noncommuting family with state-dependent jumps, including both
recurrent and transient cases.
\end{example}

\section{Effective dimension and the drift--dimension dichotomy}
\label{sec:dichotomy}

A homogeneous walk may live in a proper subgroup of its ambient lattice.
Under the nondegeneracy hypothesis on that subgroup, its rank determines
the return exponent rather than the ambient dimension \(d\).

\begin{definition}
\label{def:effective-lattice}
For a homogeneous finite-range OQRW on \(\Z^d\), let
\[
  S_{\mathrm{act}}:=\{s\in\Z^d:L_s\neq0\}
\]
and define
\begin{equation}\label{eq:effective-lattice}
  \Gamma:=\langle S_{\mathrm{act}}\rangle_{\Z}\subseteq\Z^d.
\end{equation}
Its rank
\[
  r:=\rank\Gamma
\]
is the \emph{effective dimension}.  Choose a group isomorphism
\(J\colon\Gamma\to\Z^r\).  Its inverse extends uniquely to a linear injection
\begin{equation}\label{eq:iota-effective}
  \iota\colon\R^r\longrightarrow\R^d.
\end{equation}
For a walk started at the origin, put
\[
  Y_n:=J(X_n).
\]
\end{definition}

The rank of \(\Gamma\) describes spatial support, and is not by itself the
rank of the asymptotic covariance.  Internal correlations can produce
additional covariance degeneracy even when the jumps generate \(\Gamma\).
Accordingly, positive definiteness in Assumption~\ref{ass:fourier} must still
be checked after the coordinate change.

The effective Fourier deformation is
\begin{equation}\label{eq:effective-fourier}
  \widetilde\Phi_\theta(\sigma)
  :=\sum_{s\in S_{\mathrm{act}}}
    e^{i\theta\cdot J(s)}L_s\sigma L_s^*,
  \qquad \theta\in\mathbb T^r.
\end{equation}
When \(r\geq1\), we apply Assumption~\ref{ass:fourier} to this family and
denote its drift and covariance by \(\widetilde m\) and
\(\Sigma_{\mathrm{eff}}\).  The ambient drift is
\begin{equation}\label{eq:ambient-drift}
  m:=\iota(\widetilde m)\in\R^d.
\end{equation}

\begin{theorem}
\label{thm:drift-dimension}
Let a homogeneous finite-range OQRW on \(\Z^d\) start from
\eqref{eq:localized-initial-state}, and let \(r\) be its effective dimension.

If \(r=0\), then \(X_n=0\) for all \(n\), the RMS ballistic speed is zero,
and the walk is recurrent at the origin.

If \(r\geq1\), assume that the effective Fourier family
\eqref{eq:effective-fourier} satisfies Assumption~\ref{ass:fourier}.  Then
\begin{equation}\label{eq:effective-speed}
  v_{\mathrm{bal}}(\rho^{(0)})=|m|.
\end{equation}
Moreover:
\begin{enumerate}[label=\textup{(\roman*)},leftmargin=2.4em]
\item if \(m\neq0\), the walk is spatially transient;
\item if \(m=0\) and \(r\geq3\), the walk is spatially transient;
\item if \(m=0\) and \(r\in\{1,2\}\), the walk is recurrent at the origin,
and \(P_0\) is recurrent in the TOM sense.
\end{enumerate}
\end{theorem}

\begin{proof}
If \(r=0\), every active jump is zero, so the first assertion is immediate.
Suppose \(r\geq1\).  We first derive the law of large numbers directly from
the spectral expansion.  For \(u\in\R^r\),
\begin{equation}\label{eq:effective-lln-characteristic}
  \E e^{iu\cdot Y_n/n}
  =\Tr\bigl(\widetilde\Phi_{u/n}^n(\rho_0)\bigr).
\end{equation}
Let \(\widetilde\lambda_0(\theta)\) and
\(\widetilde\Pi_\theta\) be the eigenvalue branch and spectral projection
issuing from \(1\) at \(\theta=0\), and put
\[
  \widetilde a(\theta)
  :=\Tr\bigl(\widetilde\Pi_\theta(\rho_0)\bigr).
\]
As in \eqref{eq:amplitude}, \(\widetilde a\) is analytic and
\(\widetilde a(0)=1\).  The local spectral decomposition gives, for each
fixed \(u\) and all sufficiently large \(n\),
\[
  \Tr\bigl(\widetilde\Phi_{u/n}^n(\rho_0)\bigr)
  =\widetilde\lambda_0(u/n)^n\widetilde a(u/n)+O(\kappa^n)
\]
for some \(\kappa<1\).
Together with \eqref{eq:lambda-expansion}, this gives
\[
  \widetilde\lambda_0(u/n)^n
  \longrightarrow e^{iu\cdot\widetilde m},
\]
while \(\widetilde a(u/n)\to1\) and the complementary term tends to zero.
L\'evy's continuity theorem therefore implies
\begin{equation}\label{eq:effective-lln}
  \frac{Y_n}{n}\longrightarrow\widetilde m
  \quad\text{in distribution, hence in probability}.
\end{equation}
Since \(X_n=\iota(Y_n)\),
\[
  \frac{X_n}{n}\longrightarrow m
  \quad\text{in probability}.
\]
Proposition~\ref{prop:speed-from-lln} proves \eqref{eq:effective-speed}.

For a finite \(K\subset\Z^d\), set
\[
  K_{\mathrm{eff}}:=J(K\cap\Gamma)\subset\Z^r.
\]
Then
\[
  \PP(X_n\in K)=\PP(Y_n\in K_{\mathrm{eff}}).
\]
The three recurrence and transience statements now follow from
Theorem~\ref{thm:uniform-llt} and
Corollary~\ref{cor:one-component-return} applied on the effective lattice.
The TOM assertion in \textup{(iii)} follows from
Corollary~\ref{cor:tom-origin-recurrence}.
\end{proof}

The resulting drift--effective-dimension dichotomy is summarized in
Figure~\ref{fig:oqrw-drift-dimension}.

\FloatBarrier
\begin{figure}[t]
\centering
\resizebox{\textwidth}{!}{%
\begin{tikzpicture}
  \node[oqrw blue,text width=6.0cm] (root) at (0,0) {
    \textbf{Primitive finite-range OQRW}\\[-1pt]
    Fourier assumptions, drift $m$, covariance $\Sigma>0$,\\
    effective lattice rank $r$
  };

  \node[oqrw red,text width=4.8cm] (drift) at (-4.2,-2.3) {
    \textbf{Nonzero drift: $m\neq0$}\\[2pt]
    $v_{\mathrm{bal}}=|m|>0$\\
    $\mathbb P(X_n\in K)\leq C_K e^{-c_Kn}$
  };

  \node[oqrw blue,text width=4.8cm] (centered) at (3.0,-2.3) {
    \textbf{Centered regime: $m=0$}\\[2pt]
    $v_{\mathrm{bal}}=0$\\
    return behavior is decided by $r$
  };

  \node[oqrw green,text width=4.8cm] (low) at (0.55,-5.55) {
    \textbf{Low effective dimension:}\\[-1pt]
    \textbf{$r=1,2$}\\[2pt]
    $\sum_n p_n(0)=\infty$\\
    $P_0$ is recurrent in the TOM sense\\
    $\mathfrak a(x)\sim |x|/\Sigma$ $(r=1)$\\
    $\mathfrak a(x)\sim \dfrac{\log R_\Sigma(x)}
      {\pi\sqrt{\det\Sigma}}$ $(r=2)$
  };

  \node[oqrw gold,text width=4.8cm] (high) at (5.85,-5.55) {
    \textbf{Transient dimension: $r\geq3$}\\[2pt]
    $\sum_n p_n(0)<\infty$\\
    $G(x)\sim \dfrac{\Gamma(r/2-1)}
      {2\pi^{r/2}\sqrt{\det\Sigma}}
      R_\Sigma(x)^{2-r}$
  };

  \node[oqrw red,text width=4.8cm] (allr) at (-4.2,-5.05) {
    \textbf{Transient in every}\\[-1pt]
    \textbf{effective dimension}\\
    drift dominates diffusive\\spreading
  };

  \draw[oqrw arrow] (root) -- node[pos=0.58,left,fill=white,inner sep=1pt,font=\small] {$m\neq0$} (drift);
  \draw[oqrw arrow] (root) -- node[pos=0.58,right,fill=white,inner sep=1pt,font=\small] {$m=0$} (centered);
  \draw[oqrw arrow] (drift) -- (allr);
  \draw[oqrw arrow] (centered) -- node[pos=0.58,left,fill=white,inner sep=1pt,font=\small] {$r\leq2$} (low);
  \draw[oqrw arrow] (centered) -- node[pos=0.58,right,fill=white,inner sep=1pt,font=\small] {$r\geq3$} (high);

  \begin{scope}[on background layer]
    \node[draw=oqrwgray!45,dashed,rounded corners=4pt,
      fit=(root)(drift)(centered)(low)(high)(allr),inner sep=7pt] {};
  \end{scope}
\end{tikzpicture}%
}
\caption{Drift--effective-dimension classification in the primitive
one-component regime.  RMS speed detects drift, whereas recurrence also
depends on the rank of the effective lattice.}
\label{fig:oqrw-drift-dimension}
\end{figure}
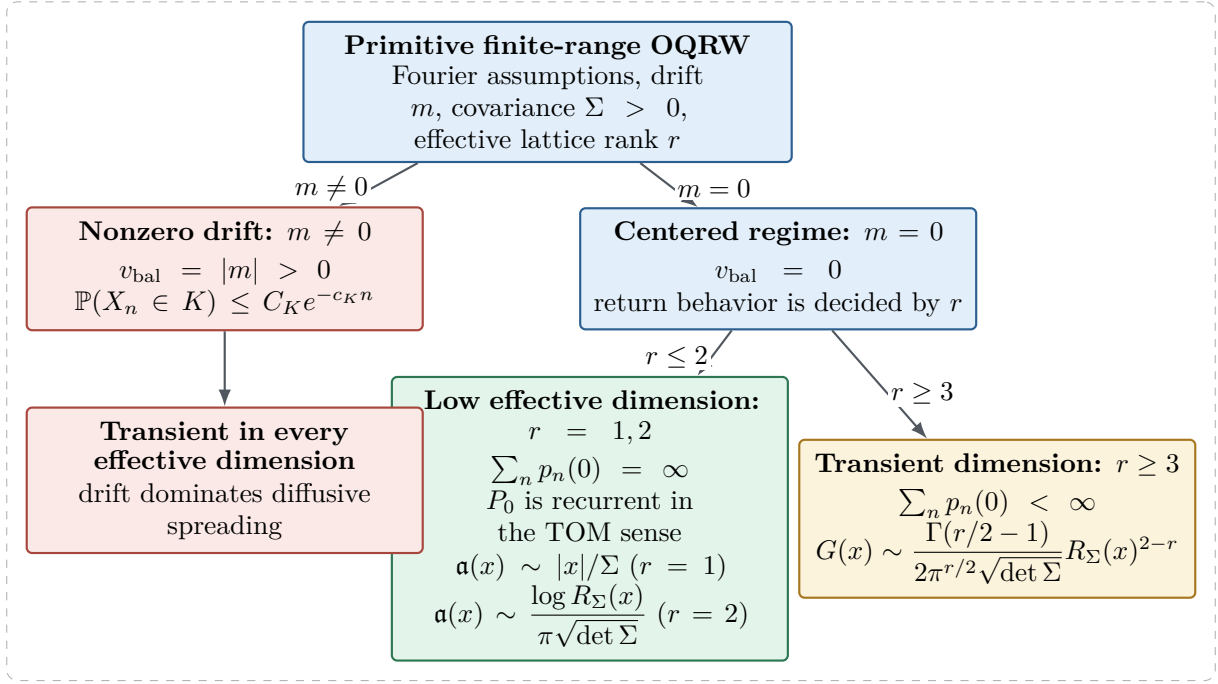
\FloatBarrier

\begin{remark}
\label{rem:ambient-effective}
A walk on \(\Z^3\) with active jumps \(0,\pm e_1\) has effective dimension
one.  In the centered periodic regime of Assumption~\ref{ass:fourier} it is
recurrent, despite the
three-dimensional ambient space.  Replacing \(r\) by \(d\) in
Theorem~\ref{thm:drift-dimension} would therefore be false.
\end{remark}

\begin{remark}
\label{rem:speed-detects}
Within the one-component regime of Theorem~\ref{thm:drift-dimension},
\[
  v_{\mathrm{bal}}>0\quad\Longleftrightarrow\quad m\neq0,
\]
and this implies transience.  Zero speed is inconclusive: centered walks are
recurrent in effective dimensions one and two and transient in dimensions at
least three.
\end{remark}

\section{Reducible walks and componentwise potentials}
\label{sec:components}

The absorption construction below gives exact identities without any
component local-limit assumption.  The later drift--dimension classification
is conditional on Assumption~\ref{ass:component-return}, which specifies
the additional return estimates.  The component velocity limits used for
the RMS formula do not by themselves supply those estimates.

We next use the recurrent-enclosure decomposition of
Carbone--Girotti--Hernandez \cite{CGH22}.  Let \(\cL\) be the local channel of
a homogeneous OQRW with finite-dimensional internal space.  Choose the finite
recurrent decomposition
\begin{equation}\label{eq:cgh-recurrent-decomposition}
  \cR=\bigoplus_{\alpha\in\mathcal I}\chi_\alpha
\end{equation}
used in their generalized central limit theorem, and write \(p_\alpha\) for
the projection onto \(\chi_\alpha\).  The absorption operator is
\begin{equation}\label{eq:absorption-operator}
  A_\alpha:=A(\chi_\alpha)
  =\lim_{n\to\infty}\cL^{*n}(p_\alpha).
\end{equation}
For the localized initial state \eqref{eq:localized-initial-state}, define
\begin{equation}\label{eq:absorption-weights}
  a_\alpha:=\Tr(A_\alpha\rho_0),
  \qquad
  \mathcal I(\rho_0):=\{\alpha\in\mathcal I:a_\alpha>0\}.
\end{equation}
The recurrent subspace \(\cR\) is the span of the supports of all invariant
states.  Since each \(\chi_\alpha\) is an enclosure, its projection is
subharmonic: \(\cL^*(p_\alpha)\geq p_\alpha\).  Thus the limit in
\eqref{eq:absorption-operator} is monotone and exists in norm, and
\(0\leq A_\alpha\leq I\), \(\cL^*(A_\alpha)=A_\alpha\).
Moreover,
\[
 A_{\cR}:=\sum_\alpha A_\alpha
 =\lim_{n\to\infty}\cL^{*n}(p_{\cR})=I.
\]
To justify the last equality, put \(B=I-A_{\cR}\).  Then \(B\geq0\) is
harmonic and vanishes on \(\cR\).  For any density matrix \(\sigma\), a
subsequence of the Ces\`aro means
\(N^{-1}\sum_{n=0}^{N-1}\cL^n(\sigma)\) converges to an invariant state
\(\tau\), by finite-dimensional compactness.  Harmonicity gives
\(\Tr(B\sigma)=\Tr(B\tau)=0\), because \(\supp\tau\subseteq\cR\).
Hence \(B=0\).  In particular, \(\sum_\alpha a_\alpha=1\).

For each accessible \(\alpha\), harmonicity makes
\(Z_n^\alpha:=\Tr(A_\alpha\rho_n)/a_\alpha\) a martingale on the
trajectory filtration \((\mathcal F_n)\).  Indeed, conditional expectation
of its next value is obtained by applying \(\cL^*\) to \(A_\alpha\).
Since \(0\leq Z_n^\alpha\leq a_\alpha^{-1}\) and
\(Z_0^\alpha=1\), it converges in \(L^1\) to a density
\(Z_\infty^\alpha\).  Set
\(d\PP^\alpha_{\rho^{(0)}}=Z_\infty^\alpha\,d\PP_{\rho^{(0)}}\),
as in \cite[Lemma~3.1]{CGH22}.  The martingale property gives
\begin{equation}\label{eq:cgh-density}
  \left.
  \frac{d\PP^\alpha_{\rho^{(0)}}}{d\PP_{\rho^{(0)}}}
  \right|_{\mathcal F_n}
  =\frac{\Tr(A_\alpha\rho_n)}{a_\alpha}.
\end{equation}
Here \(\rho_n\) is the normalized internal state along the quantum
trajectory.  If \(a_\alpha=0\), the nonnegative martingale
\(\Tr(A_\alpha\rho_n)\) has expectation zero and therefore vanishes
almost surely at every time.  Consequently
\(\sum_{\alpha\in\mathcal I(\rho_0)}a_\alpha Z_n^\alpha=1\)
almost surely.  Passing to the \(L^1\) limits proves the exact convex
decomposition
\begin{equation}\label{eq:path-measure-decomposition}
  \PP_{\rho^{(0)}}
  =\sum_{\alpha\in\mathcal I(\rho_0)}
    a_\alpha\PP^\alpha_{\rho^{(0)}}.
\end{equation}
The term ``component'' below refers to one of these probability measures; it
need not mean that the original initial state was already supported in the
corresponding enclosure.

We next give a concrete OQRW realization of each component measure.  This
also makes precise what is meant by the Fourier family associated with a
component.  Let \(q_\alpha\) be the support projection of \(A_\alpha\), put
\(\cK_\alpha:=q_\alpha\cH\), and let \(A_\alpha^{-1/2}\) denote the inverse
of \(A_\alpha^{1/2}\) on \(\cK_\alpha\), extended by zero on
\(\ker A_\alpha\).  Define
\begin{equation}\label{eq:component-h-transform}
  \widehat L_s^{(\alpha)}
  :=\left.A_\alpha^{1/2}L_sA_\alpha^{-1/2}
    \right|_{\cK_\alpha},
  \qquad
  \widehat\rho_{0,\alpha}
  :=\frac{A_\alpha^{1/2}\rho_0A_\alpha^{1/2}}{a_\alpha}.
\end{equation}
These objects are well defined.  Indeed, harmonicity
\(\cL^*(A_\alpha)=A_\alpha\) implies that \(\ker A_\alpha\) is invariant
under every \(L_s\): if \(v\in\ker A_\alpha\), then
\[
  0=\ip{v}{A_\alpha v}
   =\sum_{s\in S}\ip{L_sv}{A_\alpha L_sv},
\]
and every term in the last sum is nonnegative.  Consequently
\(q_\alpha L_s=q_\alpha L_sq_\alpha\).  Moreover,
\[
  \sum_{s\in S}\widehat L_s^{(\alpha)*}
                    \widehat L_s^{(\alpha)}
  =A_\alpha^{-1/2}\cL^*(A_\alpha)A_\alpha^{-1/2}
  =I_{\cK_\alpha},
  \qquad
  \Tr(\widehat\rho_{0,\alpha})=1.
\]
Thus \((\widehat L_s^{(\alpha)})_{s\in S}\) defines a homogeneous OQRW on
\(\cK_\alpha\).

For a word \({\bf s}=(s_1,\ldots,s_n)\), write
\(L_{\bf s}=L_{s_n}\cdots L_{s_1}\) and define
\(\widehat L_{\bf s}^{(\alpha)}\) analogously.  The relation
\(q_\alpha L_s=q_\alpha L_sq_\alpha\) gives
\[
  \widehat L_{\bf s}^{(\alpha)}
  =A_\alpha^{1/2}L_{\bf s}A_\alpha^{-1/2}
  \quad\text{on }\cK_\alpha,
\]
and hence
\begin{align*}
 &\Tr\!\left(
   \widehat L_{\bf s}^{(\alpha)}\widehat\rho_{0,\alpha}
   \widehat L_{\bf s}^{(\alpha)*}\right)\\
 &\hspace{2cm}
  =\frac1{a_\alpha}
    \Tr\bigl(A_\alpha L_{\bf s}\rho_0L_{\bf s}^*\bigr).
\end{align*}
By \eqref{eq:cgh-density}, the expression on the right is precisely the
\(\PP^\alpha_{\rho^{(0)}}\)-probability of the word \({\bf s}\).  Therefore
the component measure on jump sequences, and hence on position paths, is
exactly realized by the \(h\)-transformed OQRW
\eqref{eq:component-h-transform}.  The internal states are related by
\[
  \widehat\rho_{n,\alpha}
  =\frac{A_\alpha^{1/2}\rho_nA_\alpha^{1/2}}
         {\Tr(A_\alpha\rho_n)}
\]
for every \(n\), almost surely under \(\PP^\alpha_{\rho^{(0)}}\).
The denominator is positive almost surely under this measure, and the
transformed internal states need not equal the original internal states.

Define
\[
  \widehat S_{\mathrm{act}}^{(\alpha)}
  :=\{s\in S:\widehat L_s^{(\alpha)}\neq0\},
  \qquad
  \Gamma_\alpha
  :=\langle\widehat S_{\mathrm{act}}^{(\alpha)}\rangle_{\Z},
  \qquad
  r_\alpha:=\rank\Gamma_\alpha,
\]
and choose a group isomorphism
\(J_\alpha\colon\Gamma_\alpha\to\Z^{r_\alpha}\).  Let
\(\iota_\alpha\colon\R^{r_\alpha}\to\R^d\) be the linear extension of
\(J_\alpha^{-1}\).  For \(r_\alpha\geq1\), the effective component Fourier
family is
\begin{equation}\label{eq:component-fourier-family}
  \widehat\Phi^{(\alpha)}_\theta(\sigma)
  :=\sum_{s\in\widehat S_{\mathrm{act}}^{(\alpha)}}
    e^{i\theta\cdot J_\alpha(s)}
    \widehat L_s^{(\alpha)}\sigma
    \widehat L_s^{(\alpha)*}.
\end{equation}
This is an ordinary Fourier deformation of the form
\eqref{eq:fourier-channel}, now on \(\cK_\alpha\).

For finite \(K\subset\Z^d\), set
\begin{equation}\label{eq:component-green}
  G_K^\alpha
  :=\sum_{n=0}^\infty
    \PP^\alpha_{\rho^{(0)}}(X_n\in K).
\end{equation}

\begin{proposition}
\label{prop:green-decomposition}
For every finite \(K\subset\Z^d\),
\begin{equation}\label{eq:green-decomposition}
  G_{\rho^{(0)}}(K)
  =\sum_{\alpha\in\mathcal I(\rho_0)}a_\alpha G_K^\alpha
  \quad\text{in }[0,\infty].
\end{equation}
Equivalently,
\begin{equation}\label{eq:tom-component-potential}
  \Tr\bigl(\rho^{(0)}U(P_K)\bigr)
  =\sum_{\alpha\in\mathcal I(\rho_0)}a_\alpha G_K^\alpha.
\end{equation}
\end{proposition}

\begin{proof}
Apply \eqref{eq:path-measure-decomposition} to the event
\(\{X_n\in K\}\):
\[
  \PP_{\rho^{(0)}}(X_n\in K)
  =\sum_{\alpha\in\mathcal I(\rho_0)}
    a_\alpha\PP^\alpha_{\rho^{(0)}}(X_n\in K).
\]
Summing over \(n\) and using Tonelli's theorem proves
\eqref{eq:green-decomposition}.  Equation
\eqref{eq:tom-component-potential} follows from
\eqref{eq:green-potential}.
\end{proof}

\begin{corollary}
\label{cor:componentwise-potential}
For a fixed finite \(K\subset\Z^d\),
\begin{align}
  G_{\rho^{(0)}}(K)<\infty
  &\quad\Longleftrightarrow\quad
  G_K^\alpha<\infty
  \text{ for every }\alpha\in\mathcal I(\rho_0),
  \label{eq:component-finite-equivalence}\\
  G_{\rho^{(0)}}(K)=\infty
  &\quad\Longleftrightarrow\quad
  G_K^\alpha=\infty
  \text{ for at least one }\alpha\in\mathcal I(\rho_0).
  \label{eq:component-infinite-equivalence}
\end{align}
Thus the full walk is spatially transient if and only if every accessible
component is spatially transient, and it is recurrent at the origin if and
only if at least one accessible component is recurrent there.
\end{corollary}

\begin{proof}
All weights indexed by \(\mathcal I(\rho_0)\) are strictly positive, and the
number of components is finite.  The assertions follow immediately from the
nonnegative sum \eqref{eq:green-decomposition}.
\end{proof}

The next statement isolates exactly what is needed from component local limit
theory.

\begin{assumption}
\label{ass:component-return}
For every \(\alpha\in\mathcal I(\rho_0)\), let \(r_\alpha\) be the rank of
the effective lattice \(\Gamma_\alpha\) defined above.  Let \(m_\alpha\in\R^d\) be the component law-of-large-numbers drift:
\[
  X_n/n\longrightarrow m_\alpha
  \quad\text{in }\PP^\alpha_{\rho^{(0)}}\text{-probability}.
\]
For the decomposition chosen above these limits follow from the component
limit theorems in \cite[Theorem~3.6, Lemmas~4.2--4.3 and Theorem~4.4]{CGH22}.
Assume the following additional return properties:
\begin{enumerate}[label=\textup{(C\arabic*)},leftmargin=2.6em]
\item if \(r_\alpha=0\), then \(m_\alpha=0\) and \(X_n=0\) almost surely
under \(\PP^\alpha_{\rho^{(0)}}\);
\item if \(r_\alpha\geq1\) and \(m_\alpha\neq0\), then for every finite
\(K\subset\Z^d\) there are \(C_{K,\alpha},c_{K,\alpha}>0\) such that
\[
 \PP^\alpha_{\rho^{(0)}}(X_n\in K)
 \leq C_{K,\alpha}e^{-c_{K,\alpha}n}.
\]
If \(r_\alpha\geq1\) and \(m_\alpha=0\), assume instead that
\[
 \PP^\alpha_{\rho^{(0)}}(X_n\in K)
 \leq C_{K,\alpha}n^{-r_\alpha/2}\qquad(n\geq1)
\]
for every finite \(K\), and, when \(r_\alpha=1,2\), that for some
\(c_\alpha>0\),
\[
 \sum_{n=1}^N\PP^\alpha_{\rho^{(0)}}(X_n=0)
 \sim
 \begin{cases}
 c_\alpha\sqrt N,&r_\alpha=1,\\
 c_\alpha\log N,&r_\alpha=2.
 \end{cases}
\]
\end{enumerate}
\end{assumption}

A component with \(r_\alpha=0\) has no active transformed jump other than
zero, so the first item, including \(m_\alpha=0\), is automatic for a walk
started at the origin.
A sufficient condition for the second item is that the explicitly defined
family \eqref{eq:component-fourier-family} satisfy
Assumption~\ref{ass:fourier}.  Indeed, the component path law is exactly that
of the \(h\)-transformed OQRW, so Theorem~\ref{thm:uniform-llt} applies to it
on \(\Gamma_\alpha\), with ambient drift
\(m_\alpha=\iota_\alpha(\widetilde m_\alpha)\).  The formulation of
Assumption~\ref{ass:component-return} also covers situations where the same
return estimates are established by another method.  No primitivity or
periodicity condition on the original reducible family alone is asserted to
imply the component estimates.

The drift in Assumption~\ref{ass:component-return} is thus the same
\(m_\alpha\) as in the generalized law of large numbers below.  If the
Fourier sufficient condition applies, uniqueness of the probability limit
identifies its spectral drift with this \(m_\alpha\).

\begin{remark}[Limits of the component reduction]
\label{rem:component-scope}
The support of \(A_\alpha\) may include transient internal directions, so
the transformed channel need not be primitive.  Nor does a recurrent
restriction to \(\chi_\alpha\) by itself imply recurrence at the original
site.  For example, on \(\Z^2\) take orthonormal internal vectors \(t,r\)
and
\[
 L_{e_2}=|r\rangle\langle t|,\qquad
 L_0=2^{-1/2}|r\rangle\langle r|,\qquad
 L_{\pm e_1}=\tfrac12|r\rangle\langle r|.
\]
These operators are normalized, and the unique recurrent enclosure is
\(\C r\), with absorption operator \(I\).  Starting from
\(|t\rangle\langle t|\) at the origin, the first step is \(e_2\); all later
positions lie in \(e_2+\Z e_1\).  Consequently the origin Green series is
\(1\), although the walk is recurrent at \(e_2\).  Here the transformed
active-jump lattice has rank two, whereas the asymptotic covariance has
rank one.  This example is excluded by
Assumption~\ref{ass:component-return}; the exact potential decomposition
still holds without that assumption.
\end{remark}

\begin{theorem}
\label{thm:reducible-classification}
Under Assumption~\ref{ass:component-return}, the full walk is spatially
transient if and only if every accessible component satisfies
\begin{equation}\label{eq:component-transient-condition}
  m_\alpha\neq0
  \quad\text{or}\quad
  \bigl(m_\alpha=0\text{ and }r_\alpha\geq3\bigr).
\end{equation}
It is recurrent at the origin if and only if at least one accessible component
satisfies
\begin{equation}\label{eq:component-recurrent-condition}
  m_\alpha=0
  \quad\text{and}\quad
  r_\alpha\in\{0,1,2\}.
\end{equation}
\end{theorem}

\begin{proof}
For \(r_\alpha\geq1\), summing the exponential or polynomial upper
bounds in Assumption~\ref{ass:component-return} proves transience in the
stated regimes, while its cumulative lower asymptotics prove recurrence at
the origin in centered dimensions one and two.  A
zero-dimensional component is trapped at the origin and is therefore
recurrent.  Corollary~\ref{cor:componentwise-potential} then transfers these
component conclusions to the full walk.
\end{proof}

\begin{theorem}
\label{thm:reducible-speed}
Assume the homogeneous OQRW has finite jump set \(S\), starts from the
localized state \eqref{eq:localized-initial-state}, and satisfies the
empirical-mean conclusion of Carbone--Girotti--Hernandez
\cite[Corollary~4.5]{CGH22}:
\begin{equation}\label{eq:cgh-empirical-mean}
  \frac{X_n}{n}
  \Longrightarrow
  \sum_{\alpha\in\mathcal I(\rho_0)}a_\alpha\delta_{m_\alpha}.
\end{equation}
Then the RMS ballistic speed exists and
\begin{equation}\label{eq:reducible-speed}
  v_{\mathrm{bal}}(\rho^{(0)})^2
  =\sum_{\alpha\in\mathcal I(\rho_0)}
    a_\alpha|m_\alpha|^2.
\end{equation}
In addition,
\begin{equation}\label{eq:averaged-drift}
  \frac{\E X_n}{n}\longrightarrow
  \overline m:=\sum_{\alpha\in\mathcal I(\rho_0)}a_\alpha m_\alpha,
  \qquad
  |\overline m|\leq v_{\mathrm{bal}}(\rho^{(0)}).
\end{equation}
\end{theorem}

\begin{proof}
Let \(R=\max_{s\in S}|s|\).  Since \(X_0=0\), every trajectory satisfies
\[
  \left|\frac{X_n}{n}\right|\leq R.
\]
The limiting measure in \eqref{eq:cgh-empirical-mean} is therefore supported
in the closed ball \(\overline B(0,R)\).  Choose bounded continuous functions
on \(\R^d\) that agree on this ball with the coordinate functions and with
\(x\mapsto|x|^2\), respectively.  Applying weak convergence to these bounded
extensions gives
\[
  \frac{\E X_n}{n}\longrightarrow\sum_\alpha a_\alpha m_\alpha
\]
and
\[
  \frac{1}{n^2}\E|X_n|^2
  \longrightarrow\sum_\alpha a_\alpha|m_\alpha|^2.
\]
Taking square roots proves \eqref{eq:reducible-speed}.  The last inequality in
\eqref{eq:averaged-drift} is Jensen's inequality.
\end{proof}

Together, Proposition~\ref{prop:green-decomposition} and
Theorem~\ref{thm:reducible-speed} give the two aggregation laws illustrated
in Figure~\ref{fig:oqrw-reducible-decomposition}.

\FloatBarrier
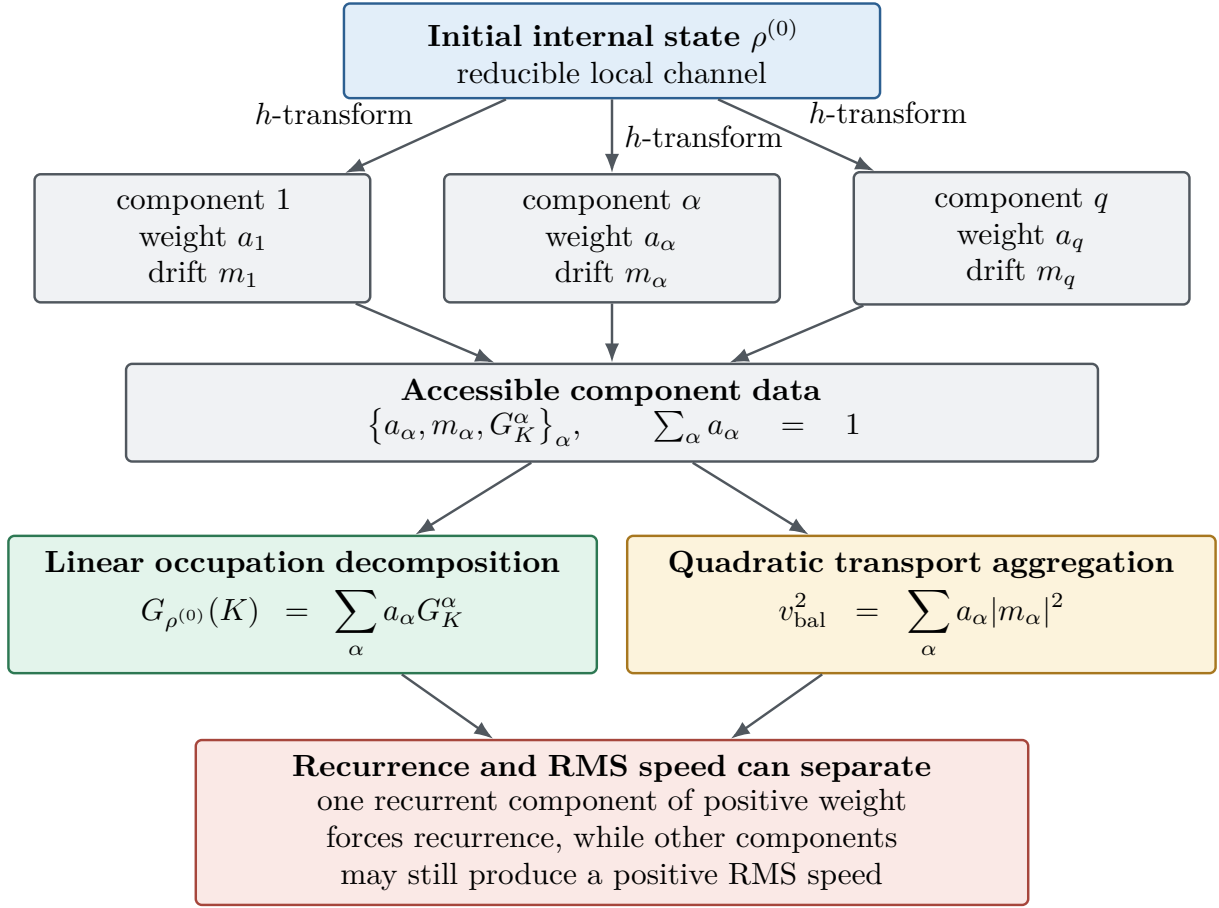
\begin{figure}[t]
\centering
\resizebox{\textwidth}{!}{%
\begin{tikzpicture}
  \node[oqrw blue,text width=5.8cm] (initial) at (0,0) {
    \textbf{Initial internal state $\rho^{(0)}$}\\
    reducible local channel
  };

  \node[oqrw gray,text width=3.5cm] (c1) at (-4.7,-2.15) {
    component $1$\\
    weight $a_1$\\
    drift $m_1$
  };
  \node[oqrw gray,text width=3.5cm] (ca) at (0,-2.15) {
    component $\alpha$\\
    weight $a_\alpha$\\
    drift $m_\alpha$
  };
  \node[oqrw gray,text width=3.5cm] (cq) at (4.7,-2.15) {
    component $q$\\
    weight $a_q$\\
    drift $m_q$
  };

  \node[oqrw gray,text width=10.8cm] (data) at (0,-4.15) {
    \textbf{Accessible component data}\\[-1pt]
    $\bigl\{a_\alpha,m_\alpha,G_K^\alpha\bigr\}_\alpha$,
    \qquad $\sum_\alpha a_\alpha=1$
  };

  \node[oqrw green,text width=6.4cm] (green) at (-3.55,-6.35) {
    \textbf{Linear occupation decomposition}\\[2pt]
    $\displaystyle
      G_{\rho^{(0)}}(K)=\sum_\alpha a_\alpha G_K^\alpha$
  };

  \node[oqrw gold,text width=6.4cm] (speed) at (3.55,-6.35) {
    \textbf{Quadratic transport aggregation}\\[2pt]
    $\displaystyle
      v_{\mathrm{bal}}^2=\sum_\alpha
      a_\alpha|m_\alpha|^2$
  };

  \node[oqrw red,text width=9.2cm] (separation) at (0,-8.85) {
    \textbf{Recurrence and RMS speed can separate}\\[-1pt]
    one recurrent component of positive weight forces recurrence, while
    other components may still produce a positive RMS speed
  };

  \draw[oqrw arrow] (initial) -- node[above left,font=\small] {$h$-transform} (c1);
  \draw[oqrw arrow] (initial) -- node[right,font=\small] {$h$-transform} (ca);
  \draw[oqrw arrow] (initial) -- node[above right,font=\small] {$h$-transform} (cq);

  \foreach \c in {c1,ca,cq}{
    \draw[oqrw arrow] (\c) -- (data);
  }
  \draw[oqrw arrow] (data) -- (green);
  \draw[oqrw arrow] (data) -- (speed);
  \draw[oqrw arrow] (green) -- (separation);
  \draw[oqrw arrow] (speed) -- (separation);
\end{tikzpicture}%
}
\caption{Two different aggregation laws in the reducible regime.  Green
occupation potentials decompose linearly, whereas RMS transport averages the
squared component drifts.  A drift--dimension criterion for the component
Green series additionally requires Assumption~\ref{ass:component-return}.}
\label{fig:oqrw-reducible-decomposition}
\end{figure}
\FloatBarrier

\begin{remark}
\label{rem:positive-speed-reducible}
Equation \eqref{eq:reducible-speed} shows that positive RMS speed means only
that at least one accessible component has nonzero drift.  Transience requires
\emph{every} accessible component to be transient.  A recurrent centered
component of positive absorption weight therefore makes the full walk
recurrent even when another component is ballistic and the RMS speed is
positive.  Also, cancellation between component drifts can make
\(\overline m=0\) while \(v_{\mathrm{bal}}>0\).
\end{remark}

\section{Exact trapping regions and sparse barriers}
\label{sec:traps}

The preceding sections concern homogeneous asymptotic regimes.  Exact traps
are instead most naturally formulated for the general TOM of
Section~\ref{sec:setup}.

\begin{definition}
A set \(F\subset\Z^d\) is \emph{forward invariant} if
\begin{equation}\label{eq:no-escape}
  B_j^i=0
  \qquad(j\in F,\ i\notin F).
\end{equation}
It is \emph{reducing} if, in addition,
\begin{equation}\label{eq:no-inflow}
  B_j^i=0
  \qquad(j\notin F,\ i\in F).
\end{equation}
\end{definition}

\begin{proposition}
\label{prop:finite-trap}
Let \(F\subset\Z^d\) be nonempty, finite, and forward invariant.  Then:
\begin{enumerate}[label=\textup{(\roman*)},leftmargin=2.4em]
\item every state initially supported in \(F\) remains supported in \(F\);
\item for every \(n\geq0\),
\begin{equation}\label{eq:trap-compression}
  P_FT^n(P_F)P_F=P_F;
\end{equation}
\item \(P_F\) is recurrent in the TOM sense;
\item every initial state supported in \(F\) has RMS ballistic speed zero.
\end{enumerate}
If \(F\) is reducing, then
\begin{equation}\label{eq:trap-harmonic}
  T(P_F)=P_F.
\end{equation}
\end{proposition}

\begin{proof}
For a source site \(j\in F\), the block of \(T(P_F)\) at \(j\) is
\[
  \sum_{i\in F}B_j^{i*}B_j^i.
\]
By \eqref{eq:no-escape} and the normalization
\eqref{eq:column-normalization}, this sum is \(I_{\cH}\).  Hence
\[
  P_FT(P_F)P_F=P_F.
\]
The blocks of \(T(P_F)\) outside \(F\) are positive, so in fact
\(T(P_F)\geq P_F\).  Positivity of \(T\) gives
\(T^n(P_F)\geq P_F\) for every \(n\), while unitality gives
\(T^n(P_F)\leq I\).  Compression by \(P_F\) proves
\eqref{eq:trap-compression}.  In the Schr\"odinger picture, the no-escape
condition directly proves preservation of support.

For \(0\neq u\in\Ran(P_F)\), \eqref{eq:trap-compression} gives
\[
  \ip{u}{T^n(P_F)u}=\norm{u}^2
  \qquad(n\geq0).
\]
The sum over \(n\) diverges, so \(P_F\) is recurrent in the TOM sense.  If the
initial position is supported in \(F\), then \(|X_n|\leq\max_{x\in F}|x|\),
which gives zero RMS speed.

Finally, under \eqref{eq:no-inflow}, the blocks of \(T(P_F)\) outside \(F\)
vanish as well, proving \eqref{eq:trap-harmonic}.
\end{proof}

\begin{remark}
Forward invariance gives recurrence of the \emph{set projection} \(P_F\), but
it does not imply recurrence of every individual site in \(F\).  For example,
the dynamics inside \(F\) may contain a transient initial segment leading to
a smaller absorbing subset.  Additional irreducibility inside the trap would
be needed for sitewise conclusions.
\end{remark}

\begin{theorem}
\label{thm:finite-cell-partition}
Assume the OQRW has finite range.  Suppose
\begin{equation}\label{eq:cell-partition}
  \Z^d=\bigsqcup_{\beta\in\mathcal B}F_\beta
\end{equation}
is a partition into finite forward-invariant cells.  Then, for every initial
state with finite position second moment,
\begin{equation}\label{eq:cell-zero-speed}
  v_{\mathrm{bal}}(\rho^{(0)})=0.
\end{equation}
Each projection \(P_{F_\beta}\) is recurrent in the TOM sense.  No uniform
bound on the diameters of the cells is required.
\end{theorem}

\begin{proof}
For each trajectory, the cell containing \(X_0\) is finite and remains
invariant.  Hence \((X_n)_{n\geq0}\) is bounded along that trajectory and
\[
  \frac{X_n}{n}\longrightarrow0
  \qquad\text{almost surely}.
\]
If the range is at most \(R\), then for \(n\geq1\),
\[
  \left|\frac{X_n}{n}\right|^2
  \leq\frac{2|X_0|^2}{n^2}+2R^2
  \leq2|X_0|^2+2R^2.
\]
The right-hand side is integrable.  Dominated convergence therefore gives
\[
  \frac{1}{n^2}\E|X_n|^2\longrightarrow0,
\]
which proves \eqref{eq:cell-zero-speed}.  Recurrence of each cell projection
follows from Proposition~\ref{prop:finite-trap}.
\end{proof}

\begin{example}
\label{ex:sparse-barriers}
Let \((b_k)_{k\in\Z}\) be a strictly increasing sequence of integers such that
\(b_k\to\pm\infty\) as \(k\to\pm\infty\).  The gaps
\(b_{k+1}-b_k\) may be unbounded.  Put
\[
  F_k:=\{b_k,b_k+1,\ldots,b_{k+1}-1\}.
\]
On the scalar internal space \(\cH=\C\), choose numbers
\(p_j^+,p_j^-\geq0\) with \(p_j^++p_j^-\leq1\), subject to
\begin{align*}
  p_{b_k}^-&=0,\\
  p_{b_{k+1}-1}^+&=0
\end{align*}
for every \(k\).  Define the only nonzero transition amplitudes by
\begin{equation}\label{eq:sparse-barrier-amplitudes}
  B_j^{j+1}=\sqrt{p_j^+},\qquad
  B_j^{j-1}=\sqrt{p_j^-},\qquad
  B_j^j=\sqrt{1-p_j^+-p_j^-}.
\end{equation}
At a one-point cell both outward probabilities are set to zero.  The
normalization is automatic, and every \(F_k\) is forward invariant.  Thus the
barriers between \(b_k-1\) and \(b_k\) are exact reflectors.  Even when the
barriers become arbitrarily sparse, Theorem~\ref{thm:finite-cell-partition}
shows that every finite-second-moment initial state has zero RMS ballistic
speed, and every cell projection is TOM-recurrent.
\end{example}

\section{Examples separating speed and recurrence}
\label{sec:examples}

The examples below are deliberately elementary.  Their purpose is to show
that recurrence cannot be read from a single second-moment growth rate.
Example~\ref{ex:centered-noncommuting} already gives centered recurrent and
transient noncommuting models within Assumption~\ref{ass:fourier}.  The
first example below illustrates why its restriction on stationary Fourier
modes must still be checked separately, even for a primitive local channel.

\begin{example}
\label{ex:primitive-zero-speed-transient}
Let \(\cH=\C^2\), and let \(\sigma_x,\sigma_y,\sigma_z\) be the Pauli
matrices.  On \(\Z^3\), set
\[
  L_0=\sqrt{\frac12}\,I,
\]
and
\[
  L_{\pm e_1}=\sqrt{\frac1{12}}\,\sigma_x,
  \qquad
  L_{\pm e_2}=\sqrt{\frac1{12}}\,\sigma_y,
  \qquad
  L_{\pm e_3}=\sqrt{\frac1{12}}\,\sigma_z.
\]
The normalization \(\sum_sL_s^*L_s=I\) is immediate.  The local channel is
primitive: writing \(X=aI+r\cdot\sigma\), one obtains
\[
  \cL(X)=aI+\frac13r\cdot\sigma,
\]
which sends every nonzero positive operator to a positive-definite one.

This example is primitive but does not satisfy all of
Assumption~\ref{ass:fourier}.  In the Pauli basis the four Fourier
eigenvalues, with \(c_j=\cos\theta_j\), are
\begin{gather*}
 \tfrac12+\tfrac16(c_1+c_2+c_3),\qquad
 \tfrac12+\tfrac16(c_1-c_2-c_3),\\
 \tfrac12+\tfrac16(-c_1+c_2-c_3),\qquad
 \tfrac12+\tfrac16(-c_1-c_2+c_3).
\end{gather*}
Thus the peripheral set consists of \(0,(0,\pi,\pi),(\pi,0,\pi),
(\pi,\pi,0)\), all with eigenvalue \(1\).  The nonzero stationary modes
act on traceless Pauli matrices, so they disappear when taking the trace.
The position law can therefore be analyzed directly, despite the failure of
the stationary-mode condition in \textup{(A2)}.

Each \(L_s\) is a scalar multiple of a unitary, so the jump probability is
independent of the internal state.  The position process is exactly the lazy
simple symmetric random walk with
\[
  \PP(\xi=0)=\frac12,
  \qquad
  \PP(\xi=\pm e_j)=\frac1{12}.
\]
Thus \(X_n=\xi_1+\cdots+\xi_n\) with independent increments,
\[
  \E X_n=0,
  \qquad
  \E|X_n|^2=\frac n2,
  \qquad
  v_{\mathrm{bal}}=0.
\]
Its one-step characteristic function is
\[
  \phi(\theta)
  =\frac12+\frac16
    \bigl(\cos\theta_1+\cos\theta_2+\cos\theta_3\bigr).
\]
The usual Fourier estimate gives
\[
  \sup_{x\in\Z^3}\PP(X_n=x)\leq Cn^{-3/2}.
\]
Hence every finite-set occupation series is summable.  This is a primitive,
full-dimensional OQRW with zero RMS speed that is spatially transient.
\end{example}

\begin{example}
\label{ex:embedded-line}
Take \(\cH=\C\) and
\[
  L_0=\frac1{\sqrt2},
  \qquad
  L_{e_1}=L_{-e_1}=\frac12,
\]
with all other jump operators zero.  The ambient space is \(\Z^3\), but the
effective lattice is \(\Z e_1\), of rank one.  The position process is the
lazy centered walk on that line, with
\[
  \E|X_n|^2=\frac n2,
  \qquad
  v_{\mathrm{bal}}=0.
\]
The one-dimensional local limit theorem gives
\[
  \PP(X_n=0)\sim\frac1{\sqrt{\pi n}},
\]
so the origin occupation series diverges.  This example shows both that zero
speed does not imply transience and that the ambient dimension is irrelevant
when the effective dimension is smaller.
\end{example}

\begin{example}
\label{ex:opposite-drifts}
Let \(\cH=\C^2\) with orthogonal rank-one projections \(P_+\) and \(P_-\),
and define on \(\Z\)
\[
  L_{+1}=P_+,
  \qquad
  L_{-1}=P_-.
\]
For the initial internal state
\[
  \rho_0=pP_++(1-p)P_-,
\]
the trajectory is \(X_n=n\) with probability \(p\) and \(X_n=-n\) with
probability \(1-p\).  Therefore
\[
  \frac{\E X_n}{n}=2p-1,
  \qquad
  \frac{\E X_n^2}{n^2}=1,
  \qquad
  v_{\mathrm{bal}}=1.
\]
For \(p=1/2\), the averaged drift is zero although the RMS speed is one.  The
two component drifts are \(+1\) and \(-1\), so
\eqref{eq:reducible-speed} gives the correct result and the walk is transient.
\end{example}

\begin{example}
\label{ex:positive-speed-recurrent}
Let \(\cH=\C^2\) with complementary orthogonal rank-one projections
\(P_b\) and \(P_r\).  Define
\[
  L_{+1}=P_b+\frac12P_r,
  \qquad
  L_{-1}=\frac12P_r,
  \qquad
  L_0=\frac1{\sqrt2}P_r.
\]
Then \(\sum_sL_s^*L_s=I\).  The \(P_b\)-component moves deterministically to
the right, while the \(P_r\)-component is the lazy centered walk with jump
probabilities \(1/4,1/2,1/4\).  Starting from
\[
  \rho_0=pP_b+(1-p)P_r,
  \qquad 0<p<1,
\]
we have
\begin{equation}\label{eq:positive-speed-recurrent-moment}
  \E X_n^2
  =pn^2+(1-p)\frac n2,
  \qquad
  v_{\mathrm{bal}}=\sqrt p>0.
\end{equation}
On the other hand, for \(n\geq1\),
\[
  \PP(X_n=0)
  =(1-p)\PP(S_n=0)
  \sim\frac{1-p}{\sqrt{\pi n}},
\]
where \(S_n\) is the lazy centered component.  The origin Green series
therefore diverges.  A positive RMS speed does not imply transience in a
reducible OQRW.  Here \(P_0\) is not TOM-recurrent: its potential on the
ballistic internal vector at the origin equals \(1\).  This distinguishes
recurrence from the specified mixed state from recurrence of the projection.
\end{example}

\section{Conclusion}
\label{sec:conclusion}
In homogeneous finite-range OQRWs, RMS ballistic speed and potential-theoretic recurrence represent separate elements of dynamics.  In a one-component regime that satisfies Assumption~\ref{ass:fourier}, the drift norm determines speed, whereas drift and effective dimension determine recurrence.  In the centered low-dimensional region, the cumulative local theorem confirms that \(P_0\) is TOM-recurrent.  For centered walks in at least three dimensions, the Green function has the anisotropic Newtonian profile \((x^T\Sigma^{-1}x)^{1-r/2}\). In dimensions one and two, the potential kernels have linear and logarithmic growth with explicit covariance constants.  In the reducible regime, the contrast is clearer. Green occupation potentials are decomposed componentwise, but the squared RMS speed is a weighted average of component drifts.  Exact finite traps provide a nonhomogeneous mechanism that enforces zero speed and TOM repetition without any mixing assumptions.

To classify reducible drift-dimensions, Assumption~\ref{ass:component-return} needs additional component return estimations.
Without local estimates, the absorption decomposition and RMS identity are still valid.
Example~\ref{ex:centered-noncommuting} provides a centered noncommuting family that meets all basic Fourier hypotheses in every dimension, including explicit Green-function and potential-kernel constants.

The phrase RMS ballistic speed relates to the uncentered second instant.  \eqref{eq:centered-ballistic-coefficient} quantifies deterministic drift in the one-component domain. Coherent unitary walks may have a nondegenerate ballistic velocity law, whereas reducible OQRWs can have a mixture of component velocities.  These processes share a macroscopic size, but do not have interchangeable recurrence requirements.

The key advancement is the transition from global scaling data to potential-theoretic categorization.  The OQRW laws of big numbers and central limit theorems may detect drifts and Gaussian fluctuations, whereas the generic TOM theory provides abstract recurrence potentials.  The periodic uniform local theorem links the two without suppressing unit-modulus Fourier modes. The explicit mask specifies the leading return amplitudes in each time residue class, while the mean-one identity produces sharp cumulative return asymptotics.  Resolvent analysis determines the spatial asymptotic profile of the potential.  These findings confirm the summability of finite-set occupations, emphasize the importance of effective dimension over ambient dimension, and support an accurate absorption decomposition in the reducible situation.  The component \(h\)-transform and paired formulae for Green potentials and RMS speed explain how transport and recurrence may concur in primitive regimes but split entirely after reduction.

The assumptions of the local limit theorem are deliberately clear.  The authors discuss local-channel mixing, the small and simple Fourier peripheral spectrum, the lack of extra stationary Fourier modes, and covariance nondegeneracy.  The theorem includes finite periodic modes despite significant aperiodicity.  Models that do not meet these assumptions need a separate return analysis. The active-jump lattice and component \(h\)-transform do not automatically provide the required hypotheses.

\section*{Funding}
A. Dhahri is a member of GNAMPA-INdAM and he has been supported by the MUR grant Dipartimento di Eccellenza 2023-2027 of
Dipartimento di Matematica, Politecnico di Milano. The work of H. J. Yoo was supported by a grant from the National Research Foundation of Korea (NRF), funded by the Korean government (MSIT) (IRIS RS-2026-25488431).

% The authors should confirm this declaration before submission.
\section*{Declaration of competing interest}
The authors declare that they have no known competing financial interests or
personal relationships that could have appeared to influence the work
reported in this paper.

\section*{Data availability}
No data were used for the research described in this article.


\begin{thebibliography}{99}

\bibitem{AhlbrechtEtAl11}
A. Ahlbrecht, H. Vogts, A. H. Werner, and R. F. Werner,
\emph{Asymptotic evolution of quantum walks with random coin},
J. Math. Phys. \textbf{52} (2011), Article 042201.
\href{https://doi.org/10.1063/1.3575568}{doi:10.1063/1.3575568}.

\bibitem{AGPS15}
S. Attal, N. Guillotin-Plantard, and C. Sabot,
\emph{Central limit theorems for open quantum random walks and quantum
measurement records},
Ann. Henri Poincar\'e \textbf{16} (2015), 15--43.

\bibitem{APSS12}
S. Attal, F. Petruccione, C. Sabot, and I. Sinayskiy,
\emph{Open quantum random walks},
J. Stat. Phys. \textbf{147} (2012), 832--852.

\bibitem{BrunCarteretAmbainis03}
T. A. Brun, H. A. Carteret, and A. Ambainis,
\emph{Quantum random walks with decoherent coins},
Phys. Rev. A \textbf{67} (2003), Article 032304.
\href{https://doi.org/10.1103/PhysRevA.67.032304}{doi:10.1103/PhysRevA.67.032304}.

\bibitem{CarbonePautrat15}
R. Carbone and Y. Pautrat,
\emph{Homogeneous open quantum random walks on a lattice},
J. Stat. Phys. \textbf{160} (2015), 1125--1153.
\href{https://doi.org/10.1007/s10955-015-1261-6}{doi:10.1007/s10955-015-1261-6}.

\bibitem{CarbonePautrat16}
R. Carbone and Y. Pautrat,
\emph{Open quantum random walks: reducibility, period, ergodic properties},
Ann. Henri Poincar\'e \textbf{17} (2016), 99--135.

\bibitem{CarvalhoGuidiLardizabal17}
S. L. Carvalho, L. F. Guidi, and C. F. Lardizabal,
\emph{Site recurrence of open and unitary quantum walks on the line},
Quantum Inf. Process. \textbf{16} (2017), Article No.~17.

\bibitem{CGH22}
R. Carbone, F. Girotti, and A. M. Hernandez,
\emph{On a generalized central limit theorem and large deviations for
homogeneous open quantum walks},
J. Stat. Phys. \textbf{188} (2022), Paper No.~8.

\bibitem{CedzichEtAl26}
C. Cedzich, A. Joye, A. H. Werner, and R. F. Werner,
\emph{Exponential tail estimates for quantum lattice dynamics},
Ann. Henri Poincar\'e \textbf{27} (2026), 2843--2870.
\href{https://doi.org/10.1007/s00023-025-01598-4}{doi:10.1007/s00023-025-01598-4}.

\bibitem{DhahriEtAl26}
A. Dhahri, C. K. Ko, F. Mukhamedov, and H. J. Yoo,
\emph{Potential theory for quantum random walks associated with transition
operation matrices},
J. Math. Anal. Appl. \textbf{562} (2026), Article 130737.

\bibitem{DhahriFagnola23}
A. Dhahri and F. Fagnola,
\emph{Potential theory for quantum Markov states and other quantum Markov
chains},
Anal. Math. Phys. \textbf{13} (2023), Article No.~31.

\bibitem{DhahriKoYoo19}
A. Dhahri, C. K. Ko, and H. J. Yoo,
\emph{Quantum Markov chains associated with open quantum random walks},
J. Stat. Phys. \textbf{176} (2019), 1272--1295.

\bibitem{DhahriMukhamedov19}
A. Dhahri and F. Mukhamedov,
\emph{Open quantum random walks, quantum Markov chains and recurrence},
Rev. Math. Phys. \textbf{31} (2019), Article 1950020.

\bibitem{DM19}
A. Dhahri and F. Mukhamedov,
\emph{Open quantum random walk and quantum Markov chains},
Funct. Anal. Appl. 53 (2019), no. 2, 137-142.


\bibitem{EvansHoeghKrohn78}
D. E. Evans and R. H\o egh-Krohn,
\emph{Spectral properties of positive maps on \(C^*\)-algebras},
J. London Math. Soc. (2) \textbf{17} (1978), 345--355.

\bibitem{GrimmettJansonScudo04}
G. Grimmett, S. Janson, and P. F. Scudo,
\emph{Weak limits for quantum random walks},
Phys. Rev. E \textbf{69} (2004), Article 026119.
\href{https://doi.org/10.1103/PhysRevE.69.026119}{doi:10.1103/PhysRevE.69.026119}.

\bibitem{GrunbaumLardizabalVelazquez20}
F. A. Gr\"unbaum, C. F. Lardizabal, and L. Vel\'azquez,
\emph{Quantum Markov chains: recurrence, Schur functions and splitting
rules},
Ann. Henri Poincar\'e \textbf{21} (2020), 189--239.

\bibitem{Gudder08}
S. Gudder,
\emph{Quantum Markov chains},
J. Math. Phys. \textbf{49} (2008), Article 072105.

\bibitem{JacqLardizabal21}
T. S. Jacq and C. F. Lardizabal,
\emph{Homogeneous open quantum walks on the line: criteria for site
recurrence and absorption},
Quantum Inf. Comput. \textbf{21} (2021), 37--58.

\bibitem{Kato95}
T. Kato,
\emph{Perturbation Theory for Linear Operators},
Classics in Mathematics, Springer, Berlin, 1995.

\bibitem{KKSY19}
C. K. Ko, N. Konno, E. Segawa, and H. J. Yoo,
\emph{Central limit theorems for open quantum random walks on the crystal
lattices},
J. Stat. Phys. \textbf{176} (2019), 710--735.

\bibitem{Konno05}
N. Konno,
\emph{A new type of limit theorems for the one-dimensional quantum random
walk},
J. Math. Soc. Japan \textbf{57} (2005), 1179--1195.
\href{https://doi.org/10.2969/jmsj/1150287309}{doi:10.2969/jmsj/1150287309}.

\bibitem{KonnoYoo13}
N. Konno and H. J. Yoo,
\emph{Limit theorems for open quantum random walks},
J. Stat. Phys. \textbf{150} (2013), 299--319.

\bibitem{MukhamedovSouissiHamdi22}
F. Mukhamedov, A. Souissi, and T. Hamdi,
\emph{Open quantum random walks and quantum Markov chains on trees I: phase
transitions},
Open Syst. Inf. Dyn. \textbf{29} (2022), Article 2250003.

\bibitem{MukhamedovSouissiHamdiAndolsi23}
F. Mukhamedov, A. Souissi, T. Hamdi, and A. A. Andolsi,
\emph{Open quantum random walks and quantum Markov chains on trees II: the
recurrence},
Quantum Inf. Process. \textbf{22} (2023), Article No.~232.

\end{thebibliography}
\end{document}